\documentclass[11pt]{amsart}
\usepackage[latin1]{inputenc}
\usepackage{amsmath}
\usepackage{amssymb}
\usepackage[all]{xy}
\usepackage{mathrsfs}
\usepackage{amsthm}

\theoremstyle{plain}
\newtheorem{thm}[equation]{Theorem}
\newtheorem{lem}[equation]{Lemma}
\newtheorem{prop}[equation]{Proposition}

\newtheorem*{theo}{Theorem}

\theoremstyle{definition}

\theoremstyle{remark}
\newtheorem{remark}[equation]{Remark}

\numberwithin{equation}{subsection}

\def\sheafEnd{\mathcal{E} \hspace{-1pt} \mathit{nd}}
\def\sheafHom{\mathcal{H} \hspace{-1pt} \mathit{om}}

\newcommand{\bk}{\Bbbk}

\newcommand{\bbD}{\mathbb{D}}

\newcommand{\bbP}{\mathbb{P}}

\newcommand{\bbZ}{\mathbb{Z}}
\newcommand{\bbX}{\mathbb{X}}
\newcommand{\bbY}{\mathbb{Y}}

\newcommand{\calA}{\mathcal{A}}
\newcommand{\calB}{\mathcal{B}}
\newcommand{\calC}{\mathcal{C}}
\newcommand{\calD}{\mathcal{D}}
\newcommand{\calE}{\mathcal{E}}
\newcommand{\calF}{\mathcal{F}}
\newcommand{\calG}{\mathcal{G}}

\newcommand{\calK}{\mathcal{K}}
\newcommand{\calL}{\mathcal{L}}
\newcommand{\calM}{\mathcal{M}}
\newcommand{\calN}{\mathcal{N}}
\newcommand{\calO}{\mathcal{O}}
\newcommand{\calP}{\mathcal{P}}
\newcommand{\calQ}{\mathcal{Q}}
\newcommand{\calR}{\mathcal{R}}
\newcommand{\calS}{\mathcal{S}}
\newcommand{\calT}{\mathcal{T}}
\newcommand{\calU}{\mathcal{U}}
\newcommand{\calV}{\mathcal{V}}

\newcommand{\wcalN}{\widetilde{\mathcal{N}}}
\newcommand{\wcalD}{\widetilde{\mathcal{D}}}

\newcommand{\fD}{\mathfrak{D}}

\newcommand{\fL}{\mathfrak{L}}

\newcommand{\fP}{\mathfrak{P}}

\newcommand{\fZ}{\mathfrak{Z}}

\newcommand{\fb}{\mathfrak{b}}

\newcommand{\fg}{\mathfrak{g}}
\newcommand{\fh}{\mathfrak{h}}

\newcommand{\fn}{\mathfrak{n}}

\newcommand{\fp}{\mathfrak{p}}

\newcommand{\ft}{\mathfrak{t}}

\newcommand{\wfg}{\widetilde{\mathfrak{g}}}

\newcommand{\bfD}{\mathbf{D}}

\newcommand{\rmD}{\mathrm{D}}

\newcommand{\rmS}{\mathrm{S}}

\newcommand{\rcap}{{\stackrel{_R}{\cap}}}
\newcommand{\Gm}{{\mathbb{G}}_{\mathbf{m}}}

\newcommand{\Hom}{\mathrm{Hom}}
\newcommand{\End}{\mathrm{End}}
\newcommand{\Ext}{\mathrm{Ext}}

\newcommand{\Coind}{{\rm Coind}}
\newcommand{\For}{{\rm For}}
\newcommand{\Id}{{\rm Id}}

\newcommand{\Coh}{{\rm Coh}}

\newcommand{\Mod}{{\rm Mod}}
\newcommand{\qis}{{\rm qis}}

\newcommand{\soc}{{\rm soc}}

\newcommand{\rk}{{\rm rk}}
\newcommand{\qc}{{\rm qc}}
\newcommand{\fig}{{\rm fg}}
\newcommand{\gr}{{\rm gr}}

\newcommand{\SL}{{\rm SL}}
\newcommand{\op}{{\rm op}}

\newcommand{\scF}{\mathscr{F}}
\newcommand{\scG}{\mathscr{G}}

\newcommand{\Ug}{\calU\fg}

\newcommand{\ri}{\rm{(i)}}
\newcommand{\rii}{\rm{(ii)}}

\newcommand{\aff}{{\rm aff}}
\newcommand{\DGCoh}{{\rm DGCoh}}

\newcommand{\HC}{{\rm HC}}

\newcommand{\Mo}{\mathrm{-Mod}_-^{\mathrm{qc}}}

\author{Simon Riche}
\address{Clermont Université, Université Blaise Pascal, Laboratoire de  
Mathéma\-tiques, BP 10448, F-63000 Clermont-Ferrand. \newline
\indent CNRS, UMR 6620, Laboratoire de Mathématiques, F-63177 Aubiere.}
\email{simon.riche@math.univ-bpclermont.fr}

\title[Koszul duality and Frobenius]{Koszul duality and Frobenius structure for restricted enveloping algebras}

\begin{document}

\begin{abstract}

Let $\fg$ be the Lie algebra of a connected, simply connected semisimple algebraic group over an algebraically closed field of sufficiently large positive characteristic. We study the compatibility between the Koszul grading on the restricted enveloping algebra $(\Ug)_0$ of $\fg$ constructed in \cite{Ri}, and the structure of Frobenius algebra of $(\Ug)_0$. This answers a question raised to the author by W.~Soergel.

\end{abstract}

\maketitle

\tableofcontents

\section*{Introduction}

\subsection{} 

Let $G$ be a connected, simply connected semisimple algebraic group over an algebraically closed field $\bk$ of characteristic $p>0$. Let \[ (\Ug)_0 \ := \ \Ug/\langle X^p - X^{[p]}, \ X \in \fg \rangle \] be the associated restricted enveloping algebra. Generalizing a result from \cite{AJS}, we have proved in \cite{Ri} that, if $p$ is sufficiently large, this algebra can be endowed with a Koszul grading, i.e.~a grading which makes it a Koszul ring in the sense of \cite{BGS}.

On the other hand, it is well-known and easy to prove (see \cite{Be}) that $(\Ug)_0$ is a Frobenius algebra. More precisely, there is a natural isomorphism of algebra $\Ug \to \Ug^{\op}$, induced by the assignment $X \in \fg \mapsto -X$, which induces an isomorphism \[ \Phi : (\Ug)_0 \ \xrightarrow{\sim} \ (\Ug)_0^{\op}. \] Using this isomorphism, the standard duality for $\bk$-vector spaces $M \mapsto M^*:=\Hom_{\bk}(M,\bk)$ induces an anti-equivalence of the category of finite dimensional (left) $(\Ug)_0$-modules, denoted $(-)^{\vee}$. Then the fact that $(\Ug)_0$ is a Frobenius algebra amounts to the existence of an isomorphism of left $(\Ug)_0$-modules \[ \theta : (\Ug)_0 \ \xrightarrow{\sim} \ \bigl( (\Ug)_0 \bigr)^{\vee}. \]

Our aim in this article is to compare these two structures. We show that there exists a Koszul grading on $(\Ug)_0$ such that $\Phi$ becomes an isomorphism of \emph{graded} algebras, and such that $\theta$ is an isomorphism of \emph{graded} $(\Ug)_0$-modules, up to some shifts to be explained below.

This problem was suggested to us by Wolfgang Soergel.

\subsection{}

To state our results precisely, we need to introduce the various blocks of the algebra $(\Ug)_0$. For simplicity, we assume from now on that $p>h$, where $h$ is the Coxeter number of $G$. Let also $T \subset B \subset G$ be a maximal torus and a Borel subgroup, $\ft \subset \fb \subset \fg$ their Lie algebras, and $\bbX:=X^*(T)$ the character lattice. Let also $W$ be the Weyl group, $W_{\aff}':=W \ltimes \bbX$ the extended affine Weyl group, and let $\rho$ be the opposite of the half sum of the roots of $\fb$. We consider the following ``dot'' actions of $W$ on $\ft^*$, respectively of $W_{\aff}'$ on $\bbX$: \[ w \bullet \lambda = w(\lambda + \rho) - \rho, \qquad \text{respectively } \ (w t_{\mu}) \bullet \nu = w(\nu + p \mu + \rho) - \rho.  \]

The subalgebra $(\Ug)^G \subset \Ug$ is central, and isomorphic to $\bk[\ft^* / (W,\bullet)]$. Its image in $(\Ug)_0$ is a central subalgebra $\fZ_0$, whose set of characters is naturally parametrized by $\bbX/(W_{\aff}',\bullet)$. For any $\lambda \in \bbX$ (or in $\bbX/(W_{\aff}',\bullet)$), we denote by $(\Ug)_0^{\widehat{\lambda}}$ the completion of $(\Ug)_0$ with respect to the annihilator of $\lambda$ in $\fZ_0$. It is a finite dimensional algebra, whose category of finitely generated modules is equivalent to the category of finitely generated $(\Ug)_0$-modules with generalized central character $\lambda$ (for $\fZ_0$). Moreover, there is a natural isomorphism of algebras \[ (\Ug)_0 \ \cong \ \prod_{\lambda \in \bbX/(W_{\aff}',\bullet)} \, (\Ug)_0^{\widehat{\lambda}}. \] Hence, constructing a Koszul grading on $(\Ug)_0$ is equivalent to constructing a Koszul grading on each on the subalgebras $(\Ug)_0^{\widehat{\lambda}}$.

For any $\lambda \in \bbX$, $\Phi$ induces an isomorphism \[ \Phi_{\lambda} : (\Ug)_0^{\widehat{\lambda}} \xrightarrow{\sim} \bigl( (\Ug)_0^{\widehat{-\lambda-2\rho}} \bigr)^{\op}, \] and $\theta$ induces an isomorphism of $(\Ug)_0^{\widehat{\lambda}}$-modules \[ \theta_{\lambda} : (\Ug)_0^{\widehat{\lambda}} \xrightarrow{\sim} \bigl( (\Ug)_0^{\widehat{-\lambda-2\rho}} \bigr)^{\vee}. \]

\subsection{}

Our main result is the following.

\begin{theo}

Assume $p \gg 0$.

There exists a Koszul grading on the algebra $(\Ug)_0$ such that: \begin{enumerate}
\item The natural isomorphism $\Phi : (\Ug)_0 \xrightarrow{\sim} \bigl((\Ug)_0\bigr)^{\op}$ is an isomorphism of graded rings.
\item For any $\lambda \in \bbX$, there exists an explicit $N_{\lambda} \in \bbZ$ and an isomorphism of graded left $(\Ug)_0^{\widehat{\lambda}}$-modules \[ (\Ug)_0^{\widehat{\lambda}} \ \cong \ \bigl((\Ug)_0^{\widehat{-\lambda-2\rho}}\bigr)^{\vee} \langle N_{\lambda} \rangle. \]

\end{enumerate}

\end{theo}

\subsection{} \label{ss:intro-shift}

Let us first explain more precisely what is $N_{\lambda}$. Recall that a weight $\lambda \in \bbX$ is called \emph{regular} if $\langle \lambda + \rho, \alpha^{\vee} \rangle \notin p \bbZ$ for any root $\alpha$. For such a weight, we have $N_{\lambda}=2 \dim(G/B)$. In this case, the theorem is proved in Propositions \ref{prop:antiautomorphism-grading} and \ref{prop:duality-regular-rep}.

For a general $\mu \in \bbX$, there exists a standard parabolic subgroup $P \subset G$, and some $w \in W_{\aff}'$, such that the stabilizer of $w \bullet \mu$ in $W_{\aff}'$ (for the dot action) is the Weyl group of the Levi of $P$ (see \cite[Lemma 1.5.2]{BMR2}). In this case, we have $N_{\mu}=2 \dim(G/P)$, and the theorem is proved in Propositions \ref{prop:antiautomorphism-grading-singular} and \ref{prop:duality-regular-rep-singular}.

\subsection{}

Now, let us explain the condition on $p$. It is the same as that of \cite{Ri}. It depends on the weight. For regular weights, the condition is that Lusztig's conjecture on characters of simple $G$-modules is true. (Recall that it is conjectured that this character formula holds as soon as $p>h$, but it is only known that it is true for $p$ bigger than an explicit bound which is much larger than $h$, see \cite{Fi}. See also \cite[\S 0.5]{Ri} for further references.) For singular weights, we make an extra assumption which is known to be true for $p$ bigger than an explicit bound depending on $G$ and the weight under consideration.

\subsection{}

As suggested by the description of $N_{\lambda}$ in \S \ref{ss:intro-shift}, the proof of the theorem (whose statement is of algebraic nature) is based on geometry, and more precisely on the localization theory in positive characteristic developed by Bezrukavnikov, Mirkovi{\'c} and Rumynin, see \cite{BMR, BMR2, BM}. We also use linear Koszul duality, a geometric version of the standard Koszul duality between symmetric and exterior algebras due to I.~Mirkovi{\'c}, see \cite[Section 2]{Ri}. (Note that here we use the \emph{covariant} version of this duality, and not the \emph{contravariant} version of \cite{MRlkd, MRHec}.) 

We first prove, in a general context, that linear Koszul duality commutes with homological duality (see Proposition \ref{prop:compatibility-lkd-duality}). Then, as in \cite{Ri}, we apply linear Koszul duality to a particular geometric context to contruct a ``Koszul duality'' between certain categories of $\Ug$-modules. The ``key result'' of \cite{Ri} states that, if $p \gg 0$, this duality sends certain (uniquely determined) lifts of simples to certain lifts of indecomposable projectives. It follows from the particular case of Proposition \ref{prop:compatibility-lkd-duality} that a geometric version of the duality $(-)^{\vee}$ commutes with our Koszul duality. We deduce that this geometric version of $(-)^{\vee}$ sends our lifts of projectives to lifts of projectives (see Proposition \ref{prop:duality-L-P}). We deduce the main theorem from this geometric statement.

\subsection{Organization of the paper}

In Section \ref{sec:lkd} we prove our general results on linear Koszul duality. In fact we reprove the equivalence of \cite[Theorem 2.3.10]{Ri} under weaker hypotheses, and with shorter proofs. We also construct homological dualities for the categories under consideration, and prove that linear Koszul duality commutes with homological duality.

In Section \ref{sec:reminder} we review the results of \cite{BMR,BMR2} and \cite{Ri} that will be needed. We also explain how one can explicitly construct the Koszul grading on $(\Ug)_0$, given our prefered lifts of indecomposable projectives (see Theorems \ref{thm:koszul-grading} and \ref{thm:koszul-grading-singular}).

In Section \ref{sec:koszul-duality-regular}, we prove the theorem for regular blocks. We first prove a geometric statement (see Proposition \ref{prop:duality-L-P}), and show that it implies our algebraic statements.

In Section \ref{sec:koszul-duality-singular}, we prove the theorem for singular blocks. The proof is very similar to that for regular blocks.

Finally, in Section \ref{sec:example} we explain the content of our results in the special case $G=\SL(2)$.

\subsection{Acknowledgements}

We warmly thank W.~Soergel for suggesting this problem.

The author is supported by the French National Research Agency (ANR-09-JCJC-0102-01).

\subsection{Some notation}

For $\bk$ an algebraically closed field of positive characteristic, and $Y$ a $\bk$-scheme, we denote by $Y^{(1)}$ the Frobenius twist of $Y$ (see \cite[\S 1.1]{BMR}).

For $Y$ a scheme, and $X \subset Y$ a subscheme, we denote by $\Coh_X(Y)$ the category of coherent sheaves on $Y$ which are set-theoretically supported on $X$.

For any smooth variety $X$ endowed with an action of an algebraic group $H$, we denote by \[ \bbD_X : \calD^b \Coh^H(X) \xrightarrow{\sim} \bigl( \calD^b \Coh^H(X) \bigr)^{\op} \] the equivalence given by $\bbD_X:=R\sheafHom_{\calO_X}(-,\calO_X)$. When $X=Y^{(1)}$, when no confusion can arise we sometimes write $\bbD_Y$ instead of $\bbD_{Y^{(1)}}$.

We denote by $V^*$ the dual of a vector space or more generally a vector bundle $V$, and by $\calV^{\star}:=\sheafHom_{\calO_X}(\calV,\calO_X)$ the dual of an $\calO_X$-module $\calV$, or a complex of $\calO_X$-modules.

In any category of $\Gm$-equivariant objects (in particular for graded modules over a graded algebra), we denote by $\langle 1 \rangle$ the tensor product with the $1$-dimensional $\bk^{\times}$-module given by $\Id_{\bk^{\times}}$. We denote by $\langle j \rangle$ the $j$-th power of $\langle 1 \rangle$.

\setcounter{subsection}{0}

\section{Linear Koszul duality and homological duality} \label{sec:lkd}

\subsection{Covariant linear Koszul duality} \label{ss:lkd}

In this subsection we reprove the results of \cite[Section 2]{Ri} in a more general setting (and with shorter proofs). The main new ideas are taken from \cite{Po}. A similar generalization of the main result of \cite{MRlkd} will appear in a forthcoming paper in collaboration with I. Mirkovi{\'c}.

Let $X$ be a scheme, and let $E \to X$ be a vector bundle over $X$. Let also $F \subset E$ be a subbundle, and $F^{\bot} \subset E^{*}$ be the orthogonal to $F$. Let also $\calE$ and $\calF$ be the (locally free) sheaves of sections of $E$ and $F$, and let $\calF^{\bot}$ be the orthogonal to $\calF$ inside $\calE^{\star}$. Our goal is to construct a ``Koszul duality'' equivalence between certain categories of coherent (dg-)sheaves on the dg-schemes\footnote{See \cite[\S 1.8]{Ri} for generalities on dg-schemes, and in particular on derived intersections.} $F$ and $F^{\bot} \rcap_{E^*} X$, where $X \subset E^*$ is the zero section.

Consider the $\Gm$-equivariant dg-algebras \[ \calS:=\mathrm{S}(\calF^{\star}), \quad \calR:=\mathrm{S}(\calF^{\star}), \quad \calT:=\Lambda(\calF) \] endowed with the trivial differential, where $\calF^{\star}$ is in bidegree $(2,-2)$, respectively $(0,-2)$, and $\calF$ is in bidegree $(-1,2)$. For a $\Gm$-equivariant dg-algebra $\calA$, we denote by $\calC(\calA\mathrm{-Mod}^{\qc})$ the category of $\Gm$-equivariant $\calA$-dg-modules which are $\calO_X$-quasi-coherent\footnote{Here for simplicity we restrict from the beginning to quasi-coherent dg-modules. However, this assumption is used only in the proof of Proposition \ref{prop:equivalences-fg}. Equivalently, one can put the quasi-coherency assumption on the cohomology of the dg-modules instead, as in \cite{Ri}.}, and by $\calC(\calA\Mo)$ the subcategory whose objects are bounded above for the internal degree (uniformly in the cohomological degree).

Consider the functors \[ \scF : \left\{ \begin{array}{ccc} \calC(\calS\Mo) & \to & \calC(\calT\Mo) \\ \calM & \mapsto & \calT^{\star} \otimes_{\calO_X} \calM \end{array} \right. \] and \[ \scG : \left\{ \begin{array}{ccc} \calC(\calT\Mo) & \to & \calC(\calS\Mo) \\ \calN & \mapsto & \calS \otimes_{\calO_X} \calN \end{array} \right. . \] Here, the actions and differentials on the dg-modules $\calT^{\star} \otimes_{\calO_X} \calM$ and $\calS \otimes_{\calO_X} \calN$ are defined as in \cite[\S 2.1]{Ri}.

\begin{lem}

The functors $\scF$ and $\scG$ are exact (i.e.~they send acyclic dg-modules to acyclic dg-modules, or equivalently quasi-isomorphisms to quasi-isomorphisms).

\end{lem}

\begin{proof} This proof is taken from \cite[Theorem A.1.2]{Po}. We give the proof for $\scG$; the case of $\scF$ is similar and simpler.

Let $\calN$ be an acyclic object of $\calC(\calS\Mo)$. Up to some shift, we can assume $\calN_i=0$ for $i>0$. Then we have \[ (\scG(\calN))_n \ = \ \bigoplus_{\genfrac{}{}{0pt}{}{i \leq 0, j \leq 0,}{n=i+j}} \calS_i \otimes_{\calO_X} \calN_j. \] Remark that this sum is finite. Hence the homogeneous components of $\scG(\calN)$ are obtained from the homogeneous components of $\calN$ by tensoring with the homogeneous components of $\calS$ (which are flat) and taking shifts and cones a finite number of times. Hence they are acyclic, which proves the result.\end{proof}

We denote by \[ \overline{\scF} : \calD(\calS\Mo) \to \calD(\calT\Mo), \quad \overline{\scG} : \calD(\calT\Mo) \to \calD(\calS\Mo) \] the functors induced between the corresponding derived categories.

\begin{prop}

The functors $\overline{\scF}$ and $\overline{\scG}$ are quasi-inverse equivalences of categories.

\end{prop}

\begin{proof} This proof is again taken from \cite[Theorem A.1.2]{Po}. Recall the ``generalized Koszul complex'' $\calK^{(1)}$ of \cite[\S 2.5]{MRlkd}. (Here, we are in the simpler situation where the vector bundle ``$\calV$'' of \cite{MRlkd} is zero.)

The functors $\scF$ and $\scG$ are clearly adjoint; hence so are $\overline{\scF}$ and $\overline{\scG}$ (see \cite[Lemma 13.6]{Ke}). We show that the adjunction morphism $\overline{\scG} \circ \overline{\scF} \to \Id$ is an isomorphism; the proof for the morphism $\Id \to \overline{\scF} \circ \overline{\scG}$ is similar. Let $\calM$ be an object of $\calC(\calS\Mo)$. Then the homogeneous internal degree components of the cone of the morphism $\scG \circ \scF (\calM) \to \calM$ can be obtained from the negative homogeneous internal degree components of $\calK^{(1)}$ by tensoring with the homogeneous internal degree components of $\calM$ and taking shifts and cones a finite number of times. These negative homogeneous internal degree components of $\calK^{(1)}$ are acyclic by \cite[Lemma 2.5.1]{MRlkd}. Hence it suffices to prove that the tensor product of an acyclic, bounded complex of flat $\calO_X$-modules with any complex of $\calO_X$-modules is acyclic. This fact is well-known, see e.g.~\cite[Proposition 5.7]{Sp}.\end{proof}

Finally we prove that these equivalences restrict to finitely generated objects. From now on, we assume that $X$ is noetherian. For a $\Gm$-equivariant dg-algebra $\calA$, we denote by $\calD^{\fig}(\calA\Mo)$ the subcategory of $\calD(\calA\Mo)$ whose objects have locally finitely generated cohomology (over $H(\calA)$). We let also $\calC \calF \calG^{\qc}(\calA)$ be the category of $\calO_X$-quasi-coherent, locally finitely generated $\Gm$-equivariant $\calA$-dg-modules. We denote by $\calD \calF \calG^{\qc}(\calA)$ the corresponding derived category. The following lemma can be proved as in \cite[Lemma 3.6.1]{MRlkd}.

\begin{lem} \label{lem:fg}

For $\calA=\calR, \calS$ or $\calT$, the natural inclusion $\calC \calF \calG^{\qc}(\calA) \hookrightarrow \calC(\calA\Mo)$ induces an equivalence of categories \[ \calD \calF \calG^{\qc}(\calA) \ \cong \ \calD^{\fig}(\calA\Mo). \]

\end{lem}

Recall the regrading functor \[ \xi : \calC(\calS\mathrm{-Mod}^{\qc}) \to \calC(\calR\mathrm{-Mod}^{\qc}) \] of \cite[\S 3.5]{MRlkd}. It is defined by the condition \[ \xi(\calM)^i_j := \calM^{i-j}_j, \] and it induces equivalences of triangulated categories \begin{equation} \label{eq:equivalence-xi} \calD \calF \calG^{\qc}(\calS) \cong \calD \calF \calG^{\qc}(\calR), \quad \calD^{\fig}(\calS\Mo) \cong \calD^{\fig}(\calR\Mo) \end{equation} (because it respects the condition of being bounded above for the internal grading).

\begin{prop} \label{prop:equivalences-fg}

The equivalences $\overline{\scF}$ and $\overline{\scG}$ restrict to equivalences \[ \calD^{\fig}(\calS\Mo) \ \cong \ \calD^{\fig}(\calT\Mo). \]

\end{prop}

\begin{proof} Any object of the category $\calC(\calT\Mo)$ has a finite filtration (as a $\calT$-dg-module) such that $\calT$ acts trivially on the associated graded. Hence, using Lemma \ref{lem:fg}, the category $\calD^{\fig}(\calT\Mo)$ is generated, as a triangulated category, by objects of $\calD^b \Coh^{\Gm}(X)$ (endowed with a trivial $\calT$-action), where $\Gm$ acts trivially on $X$.

On the other hand, we claim that $\calD^{\fig}(\calS\Mo)$ is generated, as a triangulated category, by objects of the form $\calS \otimes_{\calO_X} \calF$, for $\calF$ in $\calD^b \Coh^{\Gm}(X)$. Indeed, using the regrading functor, it is enough to prove the same result for $\calR$-dg-modules. Then this follows from the following general result (see \cite[p.~266, last paragraph]{CG}), using the fact that $\calD \calF \calG(\calR)$ is equivalent to the bounded derived category of $\Gm$-equivariant coherent sheaves on $F$.

\begin{prop}

Let $H$ be an algebraic group, and let $\pi : V \to Y$ be an $H$-equivariant vector bundle. Then the category $\calD^b \Coh^H(V)$ is generated, as a triangulated category, by objects of the form $\pi^* \calF$ for $\calF$ in $\Coh^H(Y)$.

\end{prop}

We have determined a set of generators $\calG_1$, respectively $\calG_2$, for the triangulated category $\calD^{\fig}(\calT\Mo)$, respectively $\calD^{\fig}(\calS\Mo)$. By definition, $\overline{\scF}$ and $\overline{\scG}$ induce equivalences bewteen the subcategories whose objects are in $\calG_1$ and $\calG_2$. Hence they induce an equivalence $\calD^{\fig}(\calT\Mo) \cong \calD^{\fig}(\calS\Mo)$.\end{proof}

Let us now introduce the following notation: \[ \DGCoh^{\gr}(F^{\bot} \rcap_{E^*} X) \ := \ \calD^{\fig}(\calT\Mo). \] Comparing Lemma \ref{lem:fg} and \cite[Proposition 3.3.4]{Ri}, we see that this category is equivalent to that denoted similarly in \cite[(2.3.8)]{Ri}. There is a natural forgetful functor \begin{equation} \label{eq:For-DGCoh-rcap} \For : \DGCoh^{\gr}(F^{\bot} \rcap_{E^*} X) \to \DGCoh(F^{\bot} \rcap_{E^*} X). \end{equation}

We define also \[ \DGCoh^{\gr}(F) \ := \ \calD^{\fig}(\calS\Mo). \] Consider the action of $\Gm$ on $F$ where $t \in \Gm$ acts by multiplication by $t^2$ in the fibers. By the second equivalence in \eqref{eq:equivalence-xi}, the category $\DGCoh^{\gr}(F)$ is equivalent to $\calD^b \Coh^{\Gm}(F)$. In particular, there is a natural forgetful functor \begin{equation} \label{eq:For-DGCoh-F} \For : \DGCoh^{\gr}(F) \to \calD^b \Coh(F). \end{equation}

Let us prove that the category $\DGCoh^{\gr}(F)$ is equivalent to the category denoted similarly in \cite[(2.3.6)]{Ri}. Consider the category $\calD^{+,\qc,\fig}_{\Gm}(X,\calS)$ of \cite{Ri}. By \cite[(2.3.5)]{Ri}, there is a natural functor $\phi : \calD^{+,\qc,\fig}_{\Gm}(X,\calS) \to \calD^{\fig}(\calS\Mo)$. Moreover, the following diagram commutes: \[ \xymatrix@C=40pt{\calD^{+,\qc,\fig}_{\Gm}(X,\calS) \ar[rd]^-{\psi} \ar[d]_-{\phi} & \\ \calD^{\fig}(\calS\Mo) \ar[r]^-{\overline{\scF}} & \DGCoh^{\gr}(F^{\bot} \rcap_{E^*} X), } \] where $\psi$ denotes the equivalence of \cite[(2.3.1)]{Ri}. As both $\overline{\scF}$ and $\psi$ are equivalences, so is $\phi$.

Hence we have obtained the following more general version of \cite[Theorem 2.3.10]{Ri}. Note also that the new definition of the category $\DGCoh^{\gr}(F)$ would allow to simplify (and generalize) the constructions of \cite[\S\S 2.4, 2.5 and 4.2]{Ri}.

\begin{thm} \label{thm:lkd}

There exists an equivalence of triangulated categories \[ \kappa : \DGCoh^{\gr}(F) \xrightarrow{\sim} \DGCoh^{\gr}(F^{\bot} \rcap_{E^*} X), \] called \emph{linear Koszul duality}.

\end{thm}

\begin{remark} \label{rk:shifts-grading} The functor $\kappa$ commutes with internal shifts. However, we have rather $\xi(\calM \langle m \rangle)=\xi(\calM)\langle m \rangle [-m]$ for $m \in \bbZ$. Hence the functor $\For$ of \eqref{eq:For-DGCoh-F} satisfies $\For(\calM \langle m \rangle) = \For(\calM)[-m]$.
\end{remark}

\subsection{Homological duality} \label{ss:duality}

For simplicity, from now on we assume that $X$ is a smooth variety over an algebraically closed field $\bk$. We have seen in \S \ref{ss:lkd} that the category $\DGCoh^{\gr}(F)$ is canonically equivalent to $\calD^b \Coh^{\Gm}(F)$. Now it is well-known that the functor \[ \bbD_F : \left\{ \begin{array}{ccc} \calD^b \Coh^{\Gm}(F) & \to & \calD^b \Coh^{\Gm}(F) \\ \calM & \mapsto & R\sheafHom_{\calO_{F}}(\calM,\calO_F) \end{array} \right. \] is an equivalence of categories, such that $\bbD_F \circ \bbD_F \cong \Id$. We denote by \[ \bbD_{\calS} : \ \DGCoh^{\gr}(F) \ \xrightarrow{\sim} \ \DGCoh^{\gr}(F) \] the induced equivalence.

Now we define a duality functor \[ \bbD_{\calT} : \ \DGCoh^{\gr}(F^{\bot} \rcap_{E^*} X) \ \xrightarrow{\sim} \ \DGCoh^{\gr}(F^{\bot} \rcap_{E^*} X). \]  Consider the functor \[ \rmD_{\calT} : \left\{ \begin{array}{ccc} \calC \calF \calG^{\qc}(\calT) & \to & \calC \calF \calG^{\qc}(\calT)^{\mathrm{op}} \\ \calM & \mapsto & \sheafHom_{\calT}(\calM,\calT) \end{array} \right. . \] Here, the left $\calT$-action is induced by the left multiplication of $\calT$ on itself.

\begin{lem} \label{lem:duality-T}

The functor $\rmD_{\calT}$ admits a left derived functor \[ \bbD_{\calT} : \DGCoh^{\gr}(F^{\bot} \rcap_{E^*} X) \to \DGCoh^{\gr}(F^{\bot} \rcap_{E^*} X), \] which is an equivalence of categories such that $\bbD_{\calT} \circ \bbD_{\calT} \cong \Id$.

\end{lem}

Before giving the proof of this lemma, we give an alternative definition of the functor $\rmD_{\calT}$. Let $n=\rk(\calF)$, and $\calL:=\Lambda^{n}(\calF)$ (considered as a line bundle on $X$, in internal degree $0$). For any $i=0, \cdots, n$, consider the morphism \[ \left\{ \begin{array}{ccc} \Lambda^{i}(\calF) \otimes_{\calO_X} \calL^{\otimes -1} & \to & \sheafHom_{\calO_X}(\Lambda^{n-i}(\calF), \calO_X) \\ (x \otimes a) & \mapsto & \bigl( y \mapsto (-1)^{n \cdot |y|} (x \wedge y) \otimes a \bigr) \end{array} . \right. \] This collection of morphisms induces an isomorphism of $\Gm$-equivariant $\calT$-dg-modules \begin{equation} \label{eq:coind-T} \calT \otimes_{\calO_X} \calL^{\otimes -1} \ \xrightarrow{\sim} \ \Coind_{\calT}(\calO_X)[n] \langle 2n \rangle \end{equation} (see \cite[\S 1.2]{Ri} for the definition of the coinduction functor). Using this isomorphism and \cite[\S 1.2]{Ri}, we obtain isomorphisms of $\calT$-dg-modules, for any $\calM$ in $\calC \calF \calG^{\gr}(\calT)$: \begin{multline} \label{eq:D_T} \rmD_{\calT} (\calM) \ \cong \ \sheafHom_{\calT}(\calM,\Coind_{\calT}(\calO_X)) \otimes_{\calO_X} \calL[n]\langle 2n \rangle \\ \cong \ \sheafHom_{\calO_X}(\calM,\calO_X) \otimes_{\calO_X} \calL[n]\langle 2n \rangle. \end{multline} Here, the $\calT$-module structure on $\sheafHom_{\calO_X}(\calM,\calO_X)$ is given by $(t \cdot \phi)(m)=(-1)^{|t| \cdot |\phi|}(t \cdot m)$.

The following lemma can be proved exactly as in \cite[Proposition 3.1.1]{MRlkd}.

\begin{lem} \label{lem:D_T-resolutions}

For every $\calM$ in $\calC \calF \calG^{\gr}(\calT)$, there exists an object $\calP$ of $\calC \calF \calG^{\gr}(\calT)$ such that for every $i$ and $j$ the $\calO_X$-module $\calP^i_j$ is locally free of finite rank, and a (surjective) quasi-isomorphism $\calP \xrightarrow{\qis} \calM$.

\end{lem}

Finally we can prove Lemma \ref{lem:duality-T}.

\begin{proof}[Proof of Lemma {\rm \ref{lem:duality-T}}] Using isomorphism \eqref{eq:D_T}, it is clear that any object $\calP$ as in Lemma \ref{lem:D_T-resolutions} is split on the left for the functor $\rmD_{\calT}$. This lemma asserts that there are enough such objects in $\calC \calF \calG^{\gr}(\calT)$, which implies the existence of the derived functor.

It is also follows from \eqref{eq:D_T} that $\rmD_{\calT} \circ \rmD_{\calT} (\calP) \cong \calP$ naturally for any object $\calP$ as in Lemma \ref{lem:D_T-resolutions}. Hence $\bbD_{\calT} \circ \bbD_{\calT} \cong \Id$.\end{proof}

\subsection{Compatibility} 

In this subsection we prove that the Koszul duality $\kappa$ is compatible with the duality equivalences $\bbD_{\calT}$ and $\bbD_{\calS}$. More precisely, we prove the following. Recall the notation $n$, $\calL$ introduced in \S \ref{ss:duality}

\begin{prop} \label{prop:compatibility-lkd-duality}

Consider the following diagram of equivalences \[ \xymatrix@C=30pt{ \DGCoh^{\gr}(F) \ar[r]^-{\kappa}_-{\sim} \ar[d]_-{\bbD_{\calS}}^-{\wr} & \DGCoh^{\gr}(F^{\bot} \rcap_{E^*} X) \ar[d]^-{\bbD_{\calT}}_-{\wr} \\ \DGCoh^{\gr}(F) \ar[r]^-{\kappa}_-{\sim} & \DGCoh^{\gr}(F^{\bot} \rcap_{E^*} X). } \] For any $\calM$ in $\DGCoh^{\gr}(F)$ there is a functorial isomorphism \[ \bbD_{\calT} \circ \kappa(\calM) \ \cong \ \kappa \circ \bbD_{\calS}(\calM) \otimes_{\calO_X} \calL \langle 2n \rangle [n]. \]

\end{prop}

\begin{proof} In fact we will rather work with the equivalence $\kappa^{-1}$. Let $\calP$ be an object of $\DGCoh^{\gr}(F^{\bot} \rcap_{E^*} X)$ which satisfies the assumptions of Lemma \ref{lem:D_T-resolutions}. Then we have, using \eqref{eq:D_T}, \[ \kappa^{-1} \circ \bbD_{\calT}(\calP) \ \cong \ \calS \otimes_{\calO_X} \sheafHom_{\calO_X}(\calP,\calO_X) \otimes_{\calO_X} \calL [n] \langle 2n \rangle. \] The differential on the right hand side is a Koszul differential. On the other hand, we have \[ \bbD_{\calS} \circ \kappa^{-1}(\calP) \cong \bbD_{\calS}(\calS \otimes_{\calO_X} \calP) \cong \calS \otimes_{\calO_X} \sheafHom_{\calO_X}(\calP,\calO_X). \] Again, the differential on the right hand side is a Koszul differential. The result follows.\end{proof}

\subsection{Duality on $F^{\bot} \rcap_{E^*} X$ and $F^{\bot}$} \label{ss:duality-Fbot}

To finish this section, we study the relation between our equivalence $\bbD_{\calT}$ and the standard duality on $F^{\bot}$.

Let us consider the functor $\bbD_{F^{\bot}} : \ \calD^b \Coh^{\Gm}(F^{\bot}) \xrightarrow{\sim} \calD^b \Coh^{\Gm}(F^{\bot})$. Let also $p: F^{\bot} \rcap_{E^*} X \to F^{\bot}$ be the natural morphism of dg-schemes, and let \[ Rp_* : \ \DGCoh^{\gr}(F^{\bot} \rcap_{E^*} X) \ \to \ \calD^b \Coh^{\Gm}(F^{\bot}) \] be the associated direct image functor (see \cite[\S 1.8]{Ri}). To make this functor more explicit, it is better to use a different realization of the category $\DGCoh^{\gr}(F^{\bot} \rcap_{E^*} X)$. Namely, using the Koszul resolution \[ \mathrm{S}(\calE) \otimes_{\calO_X} \Lambda(\calE) \ \xrightarrow{\qis} \ \calO_X, \] one can realize the dg-sheaf of functions on $F^{\bot} \rcap_{E^*} X$ as the dg-algebra \[ \calQ := \Lambda_{\calO_{F^{\bot}}} (p_{F^{\bot}}^* \calE), \] where $p_{F^{\bot}} : F^{\bot} \to X$ is the projection. (Here the differential is \emph{not} trivial.) Then the category $\DGCoh^{\gr}(F^{\bot} \rcap_{E^*} X)$ is naturally equivalent to the subcategory of the derived category of $\Gm$-equivariant quasi-coherent $\calQ$-dg-modules whose cohomology is locally finitely generated. We denote the latter category by $\calD^{\fig}(\calQ\mathrm{-Mod}^{\qc})$. One can easily check that there is well-defined equivalence \[ \bbD_{\calQ} : \calD^{\fig}(\calQ\mathrm{-Mod}^{\qc}) \to \calD^{\fig}(\calQ\mathrm{-Mod}^{\qc}) \] which is obtained as the derived functor of the functor $\calM \mapsto \sheafHom_{\calQ}(\calM,\calQ)$, and which corresponds to $\bbD_{\calT}$ under the equivalence mentioned above.

Under the equivalence $\DGCoh^{\gr}(F^{\bot} \rcap_{E^*} X) \cong \calD^{\fig}(\calQ\mathrm{-Mod}^{\qc})$, the functor $Rp_*$ is simply the restriction functor from $\calQ$-dg-modules to complexes of $\calO_{F^{\bot}}$-modules.

Let $m$ be the rank of $\calE$, and let $\calK:=\Lambda^m(\calE)$, considered as a line bundle on $X$ in internal degree $0$.

\begin{lem} \label{lem:duality-Fbot}

For any $\calM$ in $\DGCoh^{\gr}(F^{\bot} \rcap_{E^*} X)$, there is a functorial isomorphism \[ Rp_* \circ \bbD_{\calT}(\calM) \ \cong \ \bigl( \bbD_{F^{\bot}} \circ Rp_*(\calM) \bigr) \otimes_{\calO_{F^{\bot}}} p_{F^{\bot}}^* \calK[m] \langle 2m \rangle. \]

\end{lem}

\begin{proof} Using the remarks above, one can work with $\calQ$-dg-modules instead of $\calT$-dg-modules. For $\calM$ split on the left for $\bbD_{\calQ}$, we have \[ Rp_* \bbD_{\calQ}(\calM) \ \cong \ \sheafHom_{\calQ}(\calM,\calQ). \] Now, one easily checks (as for isomorphism \eqref{eq:coind-T}) that there is an isomorphism of $\calQ$-dg-modules \[ \calQ \ \cong \ \Coind_{\calQ}(\calO_{F^{\bot}}) \otimes_{\calO_{F^{\bot}}} p_{F^{\bot}}^* \calK [m] \langle 2m \rangle. \] Hence we obtain  \[ Rp_* \bbD_{\calQ}(\calM) \ \cong \ \sheafHom_{\calO_{F^{\bot}}}(\calM,\calO_{F^{\bot}}) \otimes_{\calO_{F^{\bot}}} p_{F^{\bot}}^* \calK[m] \langle 2m \rangle, \] which gives the result. \end{proof}

\begin{remark} It is natural to consider that the ``dimension'' of the dg-scheme $F^{\bot} \rcap_{E^*} X$ is $\dim(F^{\bot}) + \dim(X) - \dim(E^*)$. Hence the ``difference of dimensions'' between $F^{\bot}$ and $F^{\bot} \rcap_{E^*} X$ is $m=\rk(E)$, which makes Lemma \ref{lem:duality-Fbot} consistent with the general fact that proper direct image commutes with Grothendieck-Serre duality.
\end{remark}

\subsection{Equivariant analogues}

Let us note that if an algebraic group acts on the scheme $X$, and acts linearly on the vector bundle $E$ preserving $F$, then there are obvious equivariant analogues of all the constructions and results of this section. As we will not use these equivariant versions, we do not state them.

\section{Reminder on localization in positive characteristic} \label{sec:reminder}

\subsection{Notation} \label{ss:notation}

Let $\bk$ be an algebraically closed field of characteristic
$p>0$. Let $R$ be a root system, and $G$ the corresponding connected,
semi-simple, simply-connected
algebraic group over $\bk$. In the whole paper we assume that \[ p>h, \] where $h$ is
the Coxeter number of $G$. Let $B$ be a
Borel subgroup of $G$, $T \subset B$ a maximal torus, $U$ the unipotent radical
of $B$. Let $\fg$, $\fb$, $\ft$, $\fn$, be their respective Lie algebras. Let $R^+ \subset R$
be the positive roots, chosen as the roots in $\fg/\fb$. Let $\calB:=G/B$ be the flag variety of $G$, and
$\wcalN:=T^*\calB$ its cotangent bundle. We have the geometric description
\[ \wcalN=\{(X,gB) \in \fg^* \times \calB \mid X_{|g \cdot
  \fb}=0\}. \] We also introduce the ``extended cotangent bundle''
  \[ \wfg:=\{(X,gB) \in \fg^* \times \calB \mid X_{|g \cdot
    \fn}=0\}. \] Let $\fh$ denote the ``abstract'' Cartan
  subalgebra of $\fg$, isomorphic to
  $\fb_0/[\fb_0,\fb_0]$ for any Borel subalgebra $\fb_0$
  of $\fg$.

We denote by $\bbY:=\mathbb{Z} R$ the root latice of $R$, and by
$\bbX:=X^*(T)$ the weight lattice. Let $W$ be the Weyl group of $(G,T)$,
$W_{\aff}:=W \ltimes \bbY$ the affine Weyl group, and $W_{\aff}':=W
\ltimes \bbX$ the extended affine Weyl group. Let $\rho \in \bbX$ be the half sum of the positive roots. The ``dot-action'' of $W$ on $\ft^*$ is defined by \[ w \bullet \lambda = w(\lambda + \rho) - \rho \] (where we identify $\rho$ and its differential). Similarly, there is a dot-action of $W$ on $\ft$ obtained by duality. There is also a dot-action of $W_{\aff}'$ on $\bbX$, defined by \[ (w t_{\lambda}) \bullet \mu = w(\mu + p\lambda + \rho) - \rho. \] We set \[ C_0:=\{ \lambda \in
\bbX \mid \forall \alpha \in
R^+, \ 0 < \langle \lambda + \rho, \alpha^{\vee} \rangle < p \},\]
the set of integral weights in the fundamental
alcove. We will also consider its ``closure'' \[ \overline{C_0}:=\{ \lambda \in
\bbX \mid \forall \alpha \in
R^+, \ 0 \leq \langle \lambda + \rho, \alpha^{\vee} \rangle \leq p \}.\]

If $P \subseteq G$ is a parabolic subgroup containing $B$, $\fp$
its Lie algebra, $\fp^{{\rm u}}$ the nilpotent radical of $\fp$, and
$\calP=G/P$ the corresponding flag variety, we consider the
following analogue of the variety $\wfg$:
\[ \wfg_{\calP}:=\{(X,gP) \in \fg^* \times
\calP \mid X_{|g \cdot \fp^{{\rm u}}}=0 \}.\] In particular,
$\wfg_{\calB}=\wfg$. The quotient morphism $\pi_{\calP} : \calB
\to \calP$ induces a morphism \begin{equation} \label{eq:defpiP}
  \widetilde{\pi}_{\calP}: \wfg \to
\wfg_{\calP}.\end{equation} We also denote by
$W_{P} \subseteq W$ the Weyl group of the Levi of $P$.

For any dominant weight $\lambda$, we denote by $L(\lambda)$ the simple $G$-module with highest weight $\lambda$.

Finally, we set $N:=\dim(G/B)$, $N_{\calP}:=\dim(G/P)$, $d=\dim(\fg)$.

\subsection{Localization}

In this subsection we review the localization theory in positive characteristic developped in \cite{BMR, BMR2, BM}.

Let $\fZ$ be the center of $\calU \fg$, the enveloping
algebra of $\fg$. The subalgebra of $G$-invariants,
$\fZ_{{\rm HC}}:=(\calU \fg)^G$ is central in $\calU \fg$. This
is the ``Harish-Chandra part'' of the center, which is isomorphic
to $S(\ft)^{(W,\bullet)}$, the algebra of $W$-invariants in the
symmetric algebra of $\ft$, for the dot-action. The center
$\fZ$ also has an other part, the ``Frobenius part''
$\fZ_{{\rm Fr}}$, which is generated, as an algebra, by the elements
$X^p - X^{[p]}$ for $X \in \fg$. It is isomorphic to
$S(\fg^{(1)})$, the functions on the Frobenius twist of
$\fg^*$. Under our assumption $p>h$, there is an isomorphism
\[ \fZ_{{\rm HC}} \otimes_{\fZ_{{\rm Fr}} \cap \fZ_{{\rm HC}}}
\fZ_{{\rm Fr}}
\xrightarrow{\sim} \fZ. \] Hence, a character of $\fZ$ is
given by a compatible pair $(\nu,\chi) \in \ft^* \times
\fg^{*}{}^{(1)}$. In this paper we will only consider the case when
$\chi=0$, and $\nu \in \ft^*$ is \emph{integral}, i.e.~in the image
of the natural map $\bbX \to \ft^*$. If $\lambda \in
\bbX$, we still denote by $\lambda$ its image in $\ft^*$. We
denote the corresponding specializations by \begin{align*} (\calU
\fg)^{\lambda} \ & := \ (\calU \fg) \otimes_{\fZ_{{\rm HC}}}
\bk_{\lambda}, \\[2pt] (\calU \fg)_{0} \ & := \ (\calU \fg)
\otimes_{\fZ_{{\rm Fr}}} \bk_{0}. \end{align*}

Let $\Mod^{\fig}(\calU \fg)$ be the abelian category of finitely
generated $\calU \fg$-modules. If $\lambda \in \bbX$, we denote by
$\Mod^{\fig}_{(\lambda,0)}(\calU
\fg)$ the category of finitely generated $\calU \fg$-modules on
which $\fZ$ acts with generalized character $(\lambda,0)$.
We define similarly the categories $\Mod^{\fig}_{0}((\calU
\fg)^{\lambda})$, $\Mod^{\fig}_{\lambda}((\calU \fg)_{0})$. Hence we have inclusions \[
\xymatrix@R=0.1cm{ \Mod^{\fig}_{0}((\calU \fg)^{\lambda}) \ar@{^{(}->}[rd]
  & & \\ & \Mod^{\fig}_{(\lambda,0)}(\calU
\fg) \ar@{^{(}->}[r] & \Mod^{\fig}(\calU \fg). \\
\Mod^{\fig}_{\lambda}((\calU \fg)_{0}) \ar@{^{(}->}[ru] & & } \] Note the category $\Mod^{\fig}_{\lambda}((\calU \fg)_{0})$ is equivalent to the category of finitely generated modules over $(\Ug)_0^{\widehat{\lambda}}$, the completion of $(\Ug)_0$ with the respect to the image of the ideal of $\fZ_{\HC}$ defined by $\lambda$ (see e.g.~\cite[\S 4.4]{Ri}).

Recall that a weight $\lambda \in \bbX$ is called
\emph{regular} if, for any root $\alpha$, $\langle \lambda +
\rho, \alpha^{\vee} \rangle \notin p \mathbb{Z}$, i.e. if $\lambda$ is not on any reflection hyperplane of $W_{\aff}$ (for the dot-action). If $\mu \in \bbX$, we denote by ${\rm Stab}_{(W_{\aff},\bullet)}(\mu)$ the stabilizer of $\mu$ for the dot-action of $W_{\aff}$ on $\bbX$. Under our hypothesis $p>h$, we have $(p \bbX) \cap \bbY = p \bbY$. It follows that ${\rm Stab}_{(W_{\aff},\bullet)}(\mu)$ is also the stabilizer of $\mu$ for the action of $W_{\aff}'$ on $\bbX$.

The localization theory in positive characteristic (due to Bezrukavnikov, Mirkovi{\'c} and Rumynin) provides a geometric description of the categories of $\Ug$-modules considered above. More precisely we have (see \cite[Theorem 5.3.1]{BMR} for $\ri$, and
\cite[Theorem 1.5.1.c, Lemma 1.5.2.b]{BMR2} for $\rii$):

\begin{thm}\label{thm:thmBMR}

$\ri$ Let $\lambda \in \bbX$ be regular. There exist equivalences of triangulated categories \begin{align} \label{eq:equivBMR} \calD^b
  \Coh_{\calB^{(1)}}(\wfg^{(1)})
  \ & \cong \ \calD^b\Mod^{\fig}_{(\lambda,0)}(\calU \fg), \\ \label{eq:equivBMR2}
  \calD^b\Coh_{\calB^{(1)}}(\wcalN^{(1)}) \ & \cong \
  \calD^b\Mod^{\fig}_{0}((\calU \fg)^{\lambda}). \end{align}

$\rii$ More generally, let $\mu \in \bbX$, and let $P$ be a parabolic
subgroup of $G$ containing $B$ such that ${\rm
  Stab}_{(W_{\aff},\bullet)}(\mu)=W_P$. Let $\calP=G/P$ be the corresponding
flag variety. Then there exists an equivalence of triangulated categories
\begin{equation} \label{eq:equivBMR3} \calD^b\Coh_{\calP^{(1)}}(\wfg_{\calP}^{(1)}) \
  \cong \ \calD^b\Mod^{\fig}_{(\mu,0)}(\calU \fg). \end{equation}

\end{thm}

As in \cite{BMR}, let us consider $\wcalD:=(q_* \calD_{G/U})^T$, where $q:G/N \to \calB$ is the projection, and $T$ acts on $G/U$ via right multiplication. This algebra is an Azumaya algebra over $\wfg^{(1)} \times_{\fh^*{}^{(1)}} \fh^*$. (Here, the morphism $\fh^* \to \fh^*{}^{(1)}$ is the Artin-Schreier map, see \cite{BMR}.) This Azumaya algebra splits on the formal neighborhood of $\calB^{(1)} \times \{\lambda\}$, for any $\lambda \in \bbX$. Moreover, there are natural choices of splitting bundles. For any regular $\lambda \in \bbX$, we denote by $\calM^{\lambda}$ the splitting bundle constructed as in \cite[\S 1.3.5]{BMR2}, and by \begin{align*} \gamma_{\lambda}^{\calB} :
\calD^b\Coh_{\calB^{(1)}}(\wfg^{(1)}) \ & \xrightarrow{\sim} \
\calD^b\Mod^{\fig}_{(\lambda,0)}(\calU \fg) \\ \epsilon_{\lambda}^{\calB} :
\calD^b\Coh_{\calB^{(1)}}(\wcalN^{(1)}) \ & \xrightarrow{\sim} \
\calD^b\Mod^{\fig}_{0}((\calU \fg)^{\lambda}), \end{align*} the associated equivalences \eqref{eq:equivBMR} and \eqref{eq:equivBMR2}. Note that these equivalences depend on $\lambda$, and not only on its image in $\bbX/(W_{\aff}',\bullet)$. Note also that the projection $\wfg^{(1)} \times_{\fh^*{}^{(1)}} \fh^* \to \wfg^{(1)}$ induces an isomorphism\footnote{In fact, this property is already used to deduce equivalence \eqref{eq:equivBMR} from \cite[Theorem 5.3.1]{BMR}.} between the formal neighborhood of $\calB^{(1)} \times \{\lambda\}$ and that of $\calB^{(1)}$, hence one can consider $\calM^{\lambda}$ as a vector bundle on the formal neighborhood of the zero section in $\wfg^{(1)}$.

Similarly, for $\mu, \calP$ as in Theorem \ref{thm:thmBMR}, we
denote by $\calM^{\mu}_{\calP}$ the splitting bundle on the formal neighborhood of $\calP^{(1)} \times \{\mu\}$ in $\wfg_{\calP}^{(1)} \times_{\fh^*{}^{(1)}/W_{\calP}} \fh^*/(W_{\calP},\bullet)$ (or equivalently on the formal neighborhood of the zero section in $\wfg_{\calP}^{(1)}$) constructed as in \cite[\S 1.3.5]{BMR2}, and by \[ \gamma^{\calP}_{\mu}
: \calD^b\Coh_{\calP^{(1)}}(\wfg_{\calP}^{(1)}) \
\xrightarrow{\sim} \ \calD^b\Mod^{\fig}_{(\mu,0)}(\calU \fg) \] the
associated equivalence \eqref{eq:equivBMR3}.

Finally, the following theorem is proved in \cite[Theorem 3.4.1, Proposition 3.4.13, Theorem 3.4.14]{Ri}. As in \S \ref{ss:duality-Fbot}, we denote by $p: (\wfg \, \rcap_{\fg^* \times \calB} \, \calB)^{(1)} \to \wfg^{(1)}$ and $p_{\calP}: (\wfg_{\calP} \, \rcap_{\fg^* \times \calP} \, \calP)^{(1)} \to \wfg_{\calP}^{(1)}$ the natural morphisms of dg-schemes.

\begin{thm} \label{thm:localization-restricted}

$\ri$ Let $\lambda \in \bbX$ be regular. There exists an equivalence of
triangulated categories \[ \widehat{\gamma}_{\lambda}^{\calB} : \DGCoh((\wfg
\, \rcap_{\fg^* \times \calB} \, \calB)^{(1)}) \ \xrightarrow{\sim} \
\calD^{b} \Mod^{\fig}_{\lambda}((\calU \fg)_{0}) \] such that the following diagram commutes, where the functor ${\rm Incl}$
is induced by the inclusion $\Mod^{\fig}_{\lambda}((\calU \fg)_0)
\hookrightarrow \Mod^{\fig}_{(\lambda,0)}(\calU \fg)$: \[ \xymatrix@C=60pt{
  \DGCoh((\wfg \, \rcap_{\fg^* \times
    \calB} \, \calB)^{(1)}) \ar[d]_-{Rp_*}
  \ar[r]_-{\sim}^-{\widehat{\gamma}^{\calB}_{\lambda}} & \calD^b
  \Mod^{\fig}_{\lambda}((\calU \fg)_0) \ar[d]^-{{\rm Incl}}
   \\ \calD^b
  \Coh_{\calB^{(1)}}(\wfg^{(1)}) \ar[r]^-{\gamma^{\calB}_{\lambda}}_-{\sim} & \calD^b
  \Mod^{\fig}_{(\lambda,0)}(\calU \fg). \\ } \] 

$\rii$ Let $\mu, \calP$ be as in Theorem
{\rm \ref{thm:thmBMR}}$\rii$. There exists an equivalence of triangulated categories
\[ \widehat{\gamma}^{\calP}_{\mu}: \DGCoh((\wfg_{\calP}
\, \rcap_{\fg^* \times \calP} \, \calP)^{(1)}) \ \xrightarrow{\sim} \ \calD^b
\Mod^{\fig}_{\mu}((\calU \fg)_0) \] such that the following diagram
commutes, where ${\rm Incl}$ is induced by the inclusion
$\Mod^{\fig}_{\mu}((\calU \fg)_0) \hookrightarrow \Mod^{\fig}_{(\mu,0)}(\calU
\fg)$: \[ \xymatrix@C=60pt{ \DGCoh((\wfg_{\calP} \, \rcap_{\fg^* \times
\calP} \, \calP)^{(1)}) \ar[d]_-{R(p_{\calP})_*}
\ar[r]_-{\sim}^-{\widehat{\gamma}^{\calP}_{\mu}} & \calD^b
\Mod^{\fig}_{\mu}((\calU \fg)_0)
\ar[d]^-{{\rm Incl}} \\ \calD^b
\Coh_{\calP^{(1)}}(\wfg_{\calP}^{(1)}) \ar[r]^-{\gamma^{\calP}_{\mu}}_-{\sim} & \calD^b
\Mod^{\fig}_{(\mu,0)}(\calU \fg). \\ } \]

\end{thm}

\begin{remark} \label{rk:singular-weights}

Equivalently, the condition of Theorem \ref{thm:thmBMR}$\rii$ says that $\mu$ is on the reflection hyperplane corresponding to any simple root of $W_{P}$, but not on any reflection hyperplane of a reflection (simple or not) in $W_{\aff} - W_{P}$. With this description, it is clear that if $\mu$ satisfies this condition, then $-\mu - 2 \rho$ also satisfies it (for the same parabolic subgroup).

\end{remark}

\subsection{Geometric description of duality} \label{ss:geometric-duality}

There exists a natural isomorphism of algebras $\Ug \xrightarrow{\sim} \Ug^{\op}$, induced by the assignment $X \in \fg \mapsto -X$. Hence the duality $M \mapsto M^*$ induces a duality operation $M \mapsto M^{\vee}$ on the category of finite dimensional (left) $\Ug$-modules. This duality induces a duality between the categories $\Mod^{\fig}_{(\lambda,0)}(\calU \fg)$ and $\Mod^{\fig}_{(-\lambda-2\rho,0)}(\calU \fg)$\footnote{It is more usual to replace $-\lambda-2\rho$ by $-w_0 \lambda$. However we have $-w_0 \lambda = w_0 \bullet(-\lambda-2\rho)$, hence these weights induce the same character of the Harish-Chandra center.}, between $\Mod^{\fig}_{0}((\calU \fg)^{\lambda})$ and $\Mod^{\fig}_{0}((\calU \fg)^{-\lambda-2\rho})$, and between $\Mod^{\fig}_{\lambda}((\calU \fg)_{0})$ and $\Mod^{\fig}_{-\lambda-2\rho}((\calU \fg)_{0})$. We denote all these dualities similarly. 

We denote by $\Phi : (\Ug)_0 \xrightarrow{\sim} \bigl((\Ug)_0 \bigr)^{\op}$ the induced isomorphism. For any $\lambda \in \bbX$, it induces an isomorphism \begin{equation} \label{eq:isom-Phi}  \Phi_{\lambda} : \ (\Ug)_0^{\widehat{\lambda}} \ \xrightarrow{\sim} \ \bigl((\Ug)_0^{\widehat{-\lambda-2\rho}} \bigr)^{\op}. \end{equation}

In \cite[Section 3]{BMR2}, the authors give a geometric description of these dualities. More precisely, they prove part $\ri$ of the following proposition (see \cite[Proposition 3.0.9]{BMR2}). Part $\rii$ can be proved similarly (see Remark \ref{rk:singular-weights}). Let $\sigma : \wfg^{(1)} \to \wfg^{(1)}$ be the automorphism given by multiplication by $-1$ in the fibers of the vector bundle. We use the same notation for the similarly defined automorphism of $\wfg_{\calP}^{(1)}$.

\begin{prop} \label{prop:BMR-duality}

$\ri$ Let $\lambda \in \bbX$ be regular. Then the following diagram commutes. \[ \xymatrix@C=60pt{ \calD^b\Coh_{\calB^{(1)}}(\wfg^{(1)}) \ar[d]_-{\gamma_{\lambda}^{\calB}}^-{\wr} \ar[r]^-{\sigma^* \bbD_{\wfg}[d]} & \calD^b\Coh_{\calB^{(1)}}(\wfg^{(1)}) \ar[d]^-{\gamma_{-\lambda-2\rho}^{\calB}}_-{\wr} \\ \calD^b \Mod^{\fig}_{(\lambda,0)}(\calU \fg) \ar[r]^-{(-)^{\vee}} & \calD^b \Mod^{\fig}_{(-\lambda-2\rho,0)}(\calU \fg). } \]

$\rii$ More generally, let $\mu,P$ be as in Theorem {\rm \ref{thm:thmBMR}}$\rii$. Then the following diagram commutes. \[ \xymatrix@C=60pt{ \calD^b\Coh_{\calP^{(1)}}(\wfg_{\calP}^{(1)}) \ar[d]_-{\gamma_{\mu}^{\calP}}^-{\wr} \ar[r]^-{\sigma^* \bbD_{\wfg_{\calP}}[d]} & \calD^b\Coh_{\calP^{(1)}}(\wfg_{\calP}^{(1)}) \ar[d]^-{\gamma_{-\mu-2\rho}^{\calP}}_-{\wr} \\ \calD^b \Mod^{\fig}_{(\mu,0)}(\calU \fg) \ar[r]^-{(-)^{\vee}} & \calD^b \Mod^{\fig}_{(-\mu-2\rho,0)}(\calU \fg). } \]

\end{prop}

The same proof works also for the other categories of $\calU \fg$-modules. We obtain the following results, where again $\sigma : \wcalN^{(1)} \to \wcalN^{(1)}$, $\sigma : (\wfg \, \rcap_{\fg^* \times \calB} \, \calB)^{(1)} \to (\wfg \, \rcap_{\fg^* \times \calB} \, \calB)^{(1)}$ and $\sigma : (\wfg_{\calP} \, \rcap_{\fg^* \times \calP} \, \calP)^{(1)} \to (\wfg_{\calP} \, \rcap_{\fg^* \times \calP} \, \calP)^{(1)}$ denote multiplication by $-1$ in the fibers.

\begin{prop} \label{prop:BMR-duality-HC}

Let $\lambda \in \bbX$ be regular. The following diagram commutes. \[ \xymatrix@C=60pt{ \calD^b\Coh_{\calB^{(1)}}(\wcalN^{(1)}) \ar[d]_-{\epsilon_{\lambda}^{\calB}} \ar[r]^-{\sigma^* \bbD_{\wcalN}[2N]} & \calD^b\Coh_{\calB^{(1)}}(\wcalN^{(1)}) \ar[d]^-{\epsilon_{-\lambda-2\rho}^{\calB}} \\ \calD^b \Mod^{\fig}_{0}((\calU \fg)^{\lambda}) \ar[r]^-{(-)^{\vee}} & \calD^b \Mod^{\fig}_{0}((\calU \fg)^{-\lambda-2\rho}). } \]

\end{prop}

Recall the construction of the functor $\bbD_{\calT}$ in \S \ref{ss:duality}. Similar arguments allow to construct duality functors \begin{align*} \bbD_{\calT}^0 : \DGCoh((\wfg \, \rcap_{\fg^* \times \calB} \, \calB)^{(1)}) \ & \xrightarrow{\sim} \ \DGCoh((\wfg \, \rcap_{\fg^* \times \calB} \, \calB)^{(1)}), \\ \bbD_{\calT,\calP}^0 : \DGCoh((\wfg_{\calP}
\, \rcap_{\fg^* \times \calP} \, \calP)^{(1)}) \ & \xrightarrow{\sim} \  \DGCoh((\wfg_{\calP}
\, \rcap_{\fg^* \times \calP} \, \calP)^{(1)}). \end{align*} We have:

\begin{prop} \label{prop:BMR-duality-Fr}

$\ri$ Let $\lambda \in \bbX$ be regular. The following diagram commutes. \[ \xymatrix@C=60pt{ \DGCoh((\wfg
\, \rcap_{\fg^* \times \calB} \, \calB)^{(1)}) \ar[d]_-{\widehat{\gamma}_{\lambda}^{\calB}} \ar[r]^-{\sigma^* \bbD_{\calT}^0} & \DGCoh((\wfg
\, \rcap_{\fg^* \times \calB} \, \calB)^{(1)}) \ar[d]^-{\widehat{\gamma}_{-\lambda-2\rho}^{\calB}} \\ \calD^b \Mod^{\fig}_{\lambda}((\calU \fg)_0) \ar[r]^-{(-)^{\vee}} & \calD^b \Mod^{\fig}_{-\lambda-2\rho}((\calU \fg)_0). } \]

$\rii$ More generally, let $\mu,P$ be as in Theorem {\rm \ref{thm:thmBMR}}$\rii$. The following diagram commutes. \[ \xymatrix@C=60pt{ \DGCoh((\wfg_{\calP}
\, \rcap_{\fg^* \times \calP} \, \calP)^{(1)}) \ar[d]_-{\widehat{\gamma}_{\mu}^{\calP}} \ar[r]^-{\sigma^* \bbD_{\calT,\calP}^0} & \DGCoh((\wfg_{\calP}
\, \rcap_{\fg^* \times \calP} \, \calP)^{(1)}) \ar[d]^-{\widehat{\gamma}_{-\mu-2\rho}^{\calP}} \\ \calD^b \Mod^{\fig}_{\mu}((\calU \fg)_0) \ar[r]^-{(-)^{\vee}} & \calD^b \Mod^{\fig}_{-\mu-2\rho}((\calU \fg)_0). } \]

\end{prop}

\subsection{Localization and Koszul duality: regular case} \label{ss:localization-Koszul-regular}

Let us apply the constructions of \S \ref{ss:lkd} to the case $X=\calB^{(1)}$, $E=(\fg^* \times \calB)^{(1)}$, $F=\wcalN^{(1)}$. Under our assumptions on $p$, there exists an isomorphism of $G$-modules $\fg \cong \fg^*$. We fix such an isomorphism, and use it to identify the vector bundles $E$ and $E^*$. Under this identification, the orthogonal of $\wcalN$ is $\wfg$. Hence Theorem \ref{thm:lkd} yields an equivalence \[ \kappa_{\calB} : \ \DGCoh^{\gr}(\wcalN^{(1)}) \ \xrightarrow{\sim} \ \DGCoh^{\gr}(\bigl( \wfg \, \rcap_{\fg^* \times \calB} \, \calB \bigr)^{(1)}). \]

Now, fix a regular weight $\lambda \in \bbX$. Using Theorem \ref{thm:thmBMR}$\ri$ and Theorem \ref{thm:localization-restricted}$\ri$, we have the following situation: \[
\xymatrix{ & \DGCoh^{\gr}(\wcalN^{(1)})
  \ar[r]^-{\sim}_-{\kappa_{\calB}}
  \ar[d]_-{\For}^-{\eqref{eq:For-DGCoh-F}}
& \DGCoh^{\gr}((\wfg \, \rcap_{\fg^* \times \calB} \, \calB)^{(1)})
\ar[d]^-{\For}_-{\eqref{eq:For-DGCoh-rcap}} \\ \calD^b
\Coh_{\calB^{(1)}}(\wcalN^{(1)}) \ar@{^{(}->}[r]
\ar[d]_-{\wr}^-{\epsilon^{\calB}_{\lambda}} & \calD^b \Coh(\wcalN^{(1)}) &
\DGCoh((\wfg \, \rcap_{\fg^* \times \calB} \, \calB)^{(1)})
\ar[d]^-{\wr}_-{\widehat{\gamma}^{\calB}_{\lambda}} \\ \calD^b
\Mod^{\fig}_{0}((\calU \fg)^{\lambda}) & & \calD^b
\Mod^{\fig}_{\lambda}((\calU \fg)_0). } \] Let $y \in W_{\aff}$ be the unique element such that $y^{-1} \bullet \lambda \in C_0$. Let also \[ W^0:=\{ w \in W_{\aff}' \mid w \bullet C_0 \ \text{contains a restricted dominant weight} \}. \] It is well-known (see \cite[\S 4.4]{Ri} and references therein) that the simple objects in the categories $\Mod^{\fig}_{0}((\calU \fg)^{\lambda})$ and $\Mod^{\fig}_{\lambda}((\calU \fg)_0)$ are parametrized by $W^0$: they are the $\calU \fg$-modules induced by the $G$-modules $L(wy^{-1} \bullet \lambda)$. For any $w \in W^0$, we set \[ \fL^y_w \ := \ (\epsilon^{\calB}_{\lambda})^{-1} L(wy^{-1} \bullet \lambda) \quad \in \calD^b \Coh_{\calB^{(1)}}(\wcalN^{(1)}). \] (This object only depends on $y$, and not on $\lambda$.) Similarly, for $w \in W^0$, we denote by $P(wy^{-1} \bullet \lambda)$ the projective cover of $L(wy^{-1} \bullet \lambda)$ in $\Mod^{\fig}_{\lambda}((\calU \fg)_0)$, and we set \[ \fP^y_w \ := \ (\widehat{\gamma}^{\calB}_{\lambda})^{-1} P(wy^{-1} \bullet \lambda) \quad \in \DGCoh((\wfg \, \rcap_{\fg^* \times \calB} \, \calB)^{(1)}). \] (Again, this object only depends on $y$, and not on $\lambda$.) For simplicity, here and below, when $y=1$ we omit it from the notation.

The ``key-result'' of \cite{Ri} is the following (see \cite[Theorems 4.4.3 and 8.5.2]{Ri}). Let $\tau_0=t_{\rho} \cdot w_0 \in W_{\aff}'$, and consider the natural functor \[ \zeta : \calD^b \Coh^{\Gm}_{\calB^{(1)}}(\wcalN^{(1)}) \hookrightarrow \calD^b \Coh^{\Gm}(\wcalN^{(1)}) \xrightarrow{\xi^{-1}} \DGCoh^{\gr}(\wcalN^{(1)}). \]

\begin{thm} \label{thm:Ri}

Assume $p>h$ is such that Lusztig's conjecture is true. 

There is a unique choice of lifts $\fP^{y,\gr}_v \in \DGCoh^{\gr}((\wfg \, \rcap_{\fg^* \times \calB} \, \calB)^{(1)})$ of $\fP^y_v$, respectively $\fL^{y,\gr}_v \in \calD^b \Coh^{\Gm}_{\calB^{(1)}}(\wcalN^{(1)})$ of $\fL_v^y$ ($v \in W^0$), such that for all $w \in W^0$, \begin{equation*} \kappa_{\calB}^{-1} \fP^{y,\gr}_{\tau_0 w} \ \cong \ \zeta(\fL^{y,\gr}_w)
\otimes_{\calO_{\calB^{(1)}}} \calO_{\calB^{(1)}}(-\rho) \qquad \text{in }\DGCoh^{\gr}(\wcalN^{(1)}). \end{equation*}

\end{thm}

In \cite[Theorem 9.5.1]{Ri}, we have deduced from Theorem \ref{thm:Ri} the Koszulity of regular blocks of $(\Ug)_0$. Here we give a more ``concrete'' construction of this grading. As above, fix a regular weight $\lambda \in \bbX$, and let $y \in W_{\aff}$ be such that $y^{-1} \bullet \lambda \in C_0$. We have an isomorphism of left $(\Ug)_0^{\widehat{\lambda}}$-modules \begin{equation} (\Ug)_0^{\widehat{\lambda}} \ \cong \ \bigoplus_{w \in W^0} \, P(wy^{-1} \bullet \lambda)^{\oplus n_w^{\lambda}}, \end{equation} where $n_w^{\lambda}=\dim(L(w y^{-1} \bullet \lambda))$. We fix such an isomorphism. (It can be easily checked that nothing below really depends on this choice.)

The right multiplication of $(\Ug)_0^{\widehat{\lambda}}$ on itself induces an isomorphism \[ (\Ug)_0^{\widehat{\lambda}} \ \xrightarrow{\sim} \ \End_{(\Ug)_0^{\widehat{\lambda}}}((\Ug)_0^{\widehat{\lambda}})^{\op}. \] We obtain algebra isomorphisms \begin{multline*} (\Ug)_0^{\widehat{\lambda}} \ \xrightarrow{\sim} \ \End_{(\Ug)_0^{\widehat{\lambda}}}(\bigoplus_{w \in W^0} \, P(wy^{-1} \bullet \lambda)^{\oplus n_w^{\lambda}})^{\op} \\ \underset{\sim}{\xrightarrow{(\widehat{\gamma}^{\calB}_{\lambda})^{-1}}} \ \End_{\DGCoh((\wfg \rcap_{\fg^* \times \calB} \calB)^{(1)})} \bigl( \bigoplus_{w \in W^0} \, (\fP_w^y)^{\oplus n_w^{\lambda}} \bigr)^{\op}. \end{multline*} Now the right hand side has a natural grading (as an algebra), given by \begin{multline*} \End_{\DGCoh((\wfg \rcap_{\fg^* \times \calB} \calB)^{(1)})} \bigl( \bigoplus_{w \in W^0} \, (\fP_w^y)^{\oplus n_w^{\lambda}} \bigr) \ \cong \\ \bigoplus_{m \in \bbZ} \, \Hom_{\DGCoh^{\gr}((\wfg \rcap_{\fg^* \times \calB} \calB)^{(1)})}\bigl( \bigoplus_{w \in W^0} \, (\fP_w^{y,\gr})^{\oplus n_w^{\lambda}} , \bigoplus_{w \in W^0} \, (\fP_w^{y,\gr})^{\oplus n_w^{\lambda}} \langle -m \rangle \bigr) \end{multline*} (i.e.~we consider the grading given by the natural structure of $\Gm$-equivariant algebra on a point). Hence we obtain a grading on the algebra $(\Ug)_0^{\widehat{\lambda}}$.

\begin{thm} \label{thm:koszul-grading}

Assume $p>h$ is such that Lusztig's conjecture is true. 

$\ri$ This grading only depends on the image of (the differential of) $\lambda$ in $\fh^*/(W,\bullet)$.

$\rii$ This grading makes $(\Ug)_0^{\widehat{\lambda}}$ a Koszul ring.

\end{thm}

\begin{proof} $\ri$ This follows easily from the construction of the objects $\fP^{y,\gr}_w$ in terms of the objects $\fP^{\gr}_w$, see \cite[\S 8.5]{Ri}.

$\rii$ Thanks to $\ri$, we can assume $y=1$, i.e.~$\lambda \in C_0$. Using the isomorphisms of \cite[\S 9.3]{Ri}, for any $v,w$ and $m$ we have \[ \Hom_{\DGCoh^{\gr}((\wfg \rcap_{\fg^* \times \calB} \calB)^{(1)})}\bigl( \fP_w, \fP_v \langle -m \rangle \bigr) \ \cong \ \Ext^m_{(\Ug)^0}(L(\tau_0 w \bullet \lambda),L(\tau_0 v \bullet \lambda)). \] Hence this grading is non-negative, and its $0$-part is isomorphic to \[ \prod_{w \in W^0} \mathrm{Mat}_{n_w^{\lambda}}(\bk), \] hence is semisimple. Then, by \cite[Proposition 2.1.3]{BGS}, it is enough to prove that for any simple graded modules $L_1$ and $L_2$ concentrated in (internal) degree $0$, \begin{equation} \label{eq:koszul-condition} \Ext^i_{\Mod^{\gr}((\Ug)_0^{\widehat{\lambda}})}(L_1,L_2 \langle j \rangle)=0 \quad \text{unless } i=j. \end{equation} However, $(\Ug)_0^{\widehat{\lambda}}$ is clearly Morita equivalent, as a graded ring, to the graded ring \begin{multline*} \End_{\DGCoh((\wfg \rcap_{\fg^* \times \calB} \calB)^{(1)})} \bigl( \bigoplus_{w \in W^0} \, \fP_w \bigr)^{\op} \\ \cong \ \Bigl( \bigoplus_{m \in \bbZ} \, \Ext^m_{(\Ug)^0} \bigl( \bigoplus_{w \in W^0} L(w \bullet \lambda), \bigoplus_{w \in W^0} L(w \bullet \lambda) \bigr) \Bigr)^{\op}. \end{multline*} (Here the grading on the left hand side is defined as above.) The latter graded ring is Koszul by \cite[Theorem 9.5.1 and its proof]{Ri}. Hence it satisfies condition \eqref{eq:koszul-condition}. It follows that the same is true for $(\Ug)_0^{\widehat{\lambda}}$. (Note that the Morita equivalence under consideration preserves modules concentrated in internal degree $0$.)\end{proof}

\subsection{Localization and Koszul duality: singular case}

Let $\mu,P$ be as in Theorem \ref{thm:thmBMR}$\rii$. Let also $\lambda \in \bbX$ be a regular weight such that $\langle \lambda,\alpha^{\vee} \rangle=0$ for any root $\alpha$ of the Levi of $P$, and such that $\mu$ is in the closure of the alcove of $\lambda$. Let $y \in W_{\aff}$ be such that $y^{-1} \bullet \lambda \in C_0$. Then $\mu_0:=y^{-1} \bullet \mu \in \overline{C_0}.$

As in \cite[\S 1.10]{BMR2} and \cite[\S 10.1]{Ri}, we denote by $\fD^{\lambda}_{\calP}$ the sheaf of $\lambda$-twisted differential operators on $\calP$, and we define $U^{\lambda}_{\calP}:=\Gamma(\calP,\fD^{\lambda}_{\calP})$. The action of $G$ on $G/P$ induces an algebra morphism \[ \phi_{\calP}^{\lambda} : (\Ug)^{\lambda} \to U_{\calP}^{\lambda} \] (see \cite[\S 1.10.7]{BMR2}). We consider the following two conditions: \begin{eqnarray} \label{eq:highervanishing} & R^i \Gamma(\fD_{\calP}^{\lambda})=0 \quad \text{for} \ i>0, & \\ \label{eq:phisurjective} & \phi^{\lambda}_{\calP} \ \text{is surjective.} & \end{eqnarray} It is known that both of these conditions are satisfied for $p \gg 0$ (see \cite[Footnotes 21 and 22]{Ri} for details).

The morphism $\widetilde{\pi}_{\calP} : \wfg \to \wfg_{\calP}$ induces a morphism of dg-schemes \[ \widehat{\pi}_{\calP} : (\wfg
\, \rcap_{\fg^* \times \calB} \, \calB)^{(1)} \to (\wfg_{\calP}
\, \rcap_{\fg^* \times \calP} \, \calP)^{(1)}. \] Consider the associated inverse image functors \begin{align*} L(\widehat{\pi}_{\calP})^* : \DGCoh((\wfg_{\calP}
\, \rcap_{\fg^* \times \calP} \, \calP)^{(1)}) \ & \to \ \DGCoh((\wfg
\, \rcap_{\fg^* \times \calB} \, \calB)^{(1)}), \\ L(\widehat{\pi}_{\calP,\Gm})^* : \DGCoh^{\gr}((\wfg_{\calP}
\, \rcap_{\fg^* \times \calP} \, \calP)^{(1)}) \ & \to \ \DGCoh^{\gr}((\wfg
\, \rcap_{\fg^* \times \calB} \, \calB)^{(1)}), \end{align*} see \cite[\S\S 1.7 and 5.4]{Ri}.

Recall the equivalence \[ \widehat{\gamma}^{\calP}_{\mu}: \DGCoh((\wfg_{\calP}
\, \rcap_{\fg^* \times \calP} \, \calP)^{(1)}) \ \xrightarrow{\sim} \ \calD^b
\Mod^{\fig}_{\mu}((\calU \fg)_0) \] of Theorem \ref{thm:localization-restricted}$\rii$. Let $W^0_{\mu} \subset W^0$ by the subset of elements $w$ such that $w \bullet \mu_0$ is in the upper closure of $w \bullet C_0$. The simple objects in the category $\Mod^{\fig}_{\mu}((\Ug)_0)$
are the images of the simple $G$-modules $L(w \bullet \mu_0)$ for $w \in
W^0_{\mu}$. We denote by $P(w \bullet \mu_0)$ the projective cover of
$L(w \bullet \mu_0)$. For $w \in W^0_{\mu}$, we set \[ \fP^y_{\calP,w} := (\widehat{\gamma}^{\calP}_{\mu})^{-1} P(w \bullet \mu_0) \quad \in \DGCoh((\wfg_{\calP} \, \rcap_{\fg^* \times \calP} \, \calP)^{(1)}). \]

Assume that $p>h$ is such that Lusztig's conjecture is true. Then we can consider the objects $\fP^{y,\gr}_w$ for $w \in W^0_{\mu}$. It is proved in \cite[(10.2.8)]{Ri} that for $w \in W^0_{\mu}$ there is a unique lift $\fP^{y,\gr}_{\calP,w} \in \DGCoh^{\gr}((\wfg_{\calP} \, \rcap_{\fg^* \times \calP} \, \calP)^{(1)})$ of $\fP^y_{\calP,w}$ such that \begin{equation} \label{eq:fP_calP} \fP^{y,\gr}_w \langle N - N_{\calP} \rangle \ \cong \ L(\widehat{\pi}_{\calP,\Gm})^* \, \fP^{y,\gr}_{\calP,w}. \end{equation} It is proved also in \cite[Theorem 10.2.4]{Ri} that, if moreover $p$ is such that \eqref{eq:highervanishing} and \eqref{eq:phisurjective} are satisfied, these lifts have properties similar to those of Theorem \ref{thm:Ri}. In this paper we will rather use the following characterization.

\begin{lem} \label{lem:characterization-fP_calP}

Assume that $p$ is such that Lusztig's conjecture is true.

For $w \in W^0_{\mu}$, $\fP^{y,\gr}_{\calP,w}$ is the only object of $\DGCoh^{\gr}((\wfg_{\calP} \, \rcap_{\fg^* \times \calP} \, \calP)^{(1)})$ (up to isomorphism) such that \eqref{eq:fP_calP} is satisfied.

\end{lem}

\begin{proof} We have already explained that $\fP^{y,\gr}_{\calP,w}$ satisfies this condition. Now it is enough to prove that any object $\fP$ of $\DGCoh((\wfg_{\calP} \, \rcap_{\fg^* \times \calP} \, \calP)^{(1)})$ such that $\fP^{y}_w \cong L(\widehat{\pi}_{\calP})^* \fP$ is isomorphic to $\fP^y_{\calP,w}$. Consider $P:=\widehat{\pi}^{\calP}_{\mu}(\fP)$. Then by assumption and \cite[Equations (10.2.6) and (10.2.7)]{Ri} we have $T_{\mu}^{\lambda}(P) \cong P(w \bullet \lambda_0)$ (where $T_{\mu}^{\lambda}$ is the translation functor, see \cite[\S 4.3]{Ri}). Hence $P$ is projective, i.e.~a direct sum of the $P(v \bullet \mu_0)$'s ($v \in W^0_{\mu}$). As \cite[Equation (10.2.7)]{Ri} is satisfied by every such $v$, we conclude that $P=P(w \bullet \mu_0)$. This finishes the proof.\end{proof}

As in \S \ref{ss:localization-Koszul-regular}, the choice of these lifts provides a grading on the algebra $(\Ug)_0^{\widehat{\mu}}$. And the same proof as that of Theorem \ref{thm:koszul-grading} gives the following.

\begin{thm} \label{thm:koszul-grading-singular}

Assume $p>h$ is such that Lusztig's conjecture is true, and conditions \eqref{eq:highervanishing} and \eqref{eq:phisurjective} are satisfied. 

This grading makes $(\Ug)_0^{\widehat{\mu}}$ a Koszul ring.

\end{thm}

\section{Koszul duality and ordinary duality: regular case} \label{sec:koszul-duality-regular}

\subsection{Geometry}

Let us define the various geometric duality functors we are going to use. Recall the definition of $\sigma$ in \S \ref{ss:geometric-duality}. We define \begin{align*} \bfD_{\calT}^{\gr} & := \sigma^* \bbD_{\calT} \langle -3 N \rangle : \DGCoh^{\gr}((\wfg \rcap_{\fg^* \times \calB} \calB)^{(1)}) \xrightarrow{\sim} \DGCoh^{\gr}((\wfg \rcap_{\fg^* \times \calB} \calB)^{(1)}), \\ \bfD_{\calS}^{\gr} & := \sigma^* \bbD_{\calS}[N] \langle -N \rangle : \DGCoh^{\gr}(\wcalN^{(1)}) \xrightarrow{\sim} \DGCoh^{\gr}(\wcalN^{(1)}), \\ \bfD_{\wcalN}^{\gr} & := \sigma^* \bbD_{\wcalN^{(1)}}[2N] \langle -N \rangle : \calD^b \Coh^{\Gm}(\wcalN^{(1)}) \xrightarrow{\sim} \calD^b \Coh^{\Gm}(\wcalN^{(1)}), \\ \bfD_{\wfg}^{\gr} & := \sigma^* \bbD_{\wfg^{(1)}}[d]\langle 2 d - 3N \rangle : \calD^b \Coh^{\Gm}(\wfg^{(1)}) \xrightarrow{\sim} \calD^b \Coh^{\Gm}(\wfg^{(1)}). \end{align*} Here, the dualities $\bbD_{\calT}$ and $\bbD_{\calS}$ are defined as in \S \ref{ss:duality}, for our choice $X=\calB^{(1)}$, $E=(\fg^* \times \calB)^{(1)}$, $F=\wcalN^{(1)}$. With these definitions, in the following diagram the vertical lines commute: \[
\xymatrix{ \DGCoh^{\gr}(\wcalN^{(1)})
  \ar@(ur,ul)[]_-{\bfD_{\calS}^{\gr}} \ar[r]^-{\sim}_-{\kappa_{\calB}}
  \ar[d]_-{\xi}
& \DGCoh^{\gr}((\wfg \, \rcap_{\fg^* \times \calB} \, \calB)^{(1)}) \ar@(ur,ul)[]_-{\bfD_{\calT}^{\gr}}
\ar[d]^-{Rp_*} \\ \calD^b \Coh^{\Gm}(\wcalN^{(1)}) \ar@(dl,dr)[]_-{\bfD_{\wcalN}^{\gr}} &
\calD^b \Coh^{\Gm}(\wfg^{(1)}). \ar@(dl,dr)[]_-{\bfD_{\wfg}^{\gr}} } \] (Use Lemma \ref{lem:duality-T} and Remark \ref{rk:shifts-grading}.) Moreover, by Proposition \ref{prop:compatibility-lkd-duality} we have an isomorphism \begin{equation} \label{eq:isom-lkd-duality} \bfD_{\calT}^{\gr} \circ \kappa_{\calB} \ \cong \ (\kappa \circ \bfD_{\calS}^{\gr}) \otimes_{\calO_{\calB^{(1)}}} \calO_{\calB^{(1)}}(-2\rho). \end{equation}

Our main geometric result is the following. Fix a regular $\lambda \in \bbX$, and let $y,z \in W_{\aff}$ be such that $\lambda \in y \bullet C_0$, $-\lambda-2\rho \in z \bullet C_0$. We denote by \[ \iota_{\lambda} : W^0 \xrightarrow{\sim} W^0 \] the bijection such that $-w_0(w y^{-1} \bullet \lambda) = \iota_{\lambda}(w) z^{-1} \bullet (-\lambda-2\rho)$.

\begin{prop} \label{prop:duality-L-P}

Assume $p>h$ is such that Lusztig's conjecture is true.

For any $w \in W^0$, there are isomorphisms \[ \bfD_{\wcalN}^{\gr}(\fL^{y,\gr}_w) \ \cong \ \fL^{z,\gr}_{\iota_{\lambda}(w)}, \qquad \bfD_{\calT}^{\gr}(\fP^{y,\gr}_{\tau_0 w}) \ \cong \ \fP^{z, \gr}_{\tau_0 \iota_{\lambda}(w)}. \]

\end{prop}

\begin{proof} By definition of $\iota_{\lambda}$, for $w \in W^0$ there is an isomorphism \[ L(w y^{-1} \bullet \lambda)^{\vee} \cong L(\iota_{\lambda}(w) z^{-1} \bullet (-\lambda-2\rho)). \] Recall that, given a finite dimensional graded algebra $A$ and an indecomposable non-graded $A$-module $M$, there is at most one lift of $M$ as a graded $M$-module, up to isomorphism and shift in the grading (see \cite[\S 5.6]{Ri} and references therein). Hence, comparing Proposition \ref{prop:BMR-duality-HC} and the definition of $\bfD_{\wcalN}^{\gr}$, there exists $n \in \bbZ$ such that \[ \bfD_{\wcalN}^{\gr}(\fL^{y,\gr}_w) \ \cong \ \fL^{z,\gr}_{\iota_{\lambda}(w)} \langle n \rangle. \] Then we have \[ \zeta \circ \bfD_{\wcalN}^{\gr}(\fL^{y,\gr}_w) \ \cong \ \zeta(\fL^{z,\gr}_{\iota_{\lambda}(w)}) \langle n \rangle [n]. \] Using Theorem \ref{thm:Ri}, we deduce \[ \zeta \circ \bfD_{\wcalN}^{\gr}(\fL^{y,\gr}_w) \otimes_{\calB^{(1)}} \calO_{\calB^{(1)}}(-\rho) \ \cong \ \kappa_{\calB}^{-1} \fP^{z,\gr}_{\tau_0 \iota_{\lambda}(w)} \langle n \rangle [n]. \] Then, using isomorphism \eqref{eq:isom-lkd-duality}, we obtain \[ \bfD_{\calT}^{\gr}(\fP^{y,\gr}_{\tau_0 w}) \ \cong \ \fP^{z,\gr}_{\tau_0 \iota_{\lambda}(w)} \langle n \rangle [n]. \] However, by Proposition \ref{prop:BMR-duality-Fr}, the object \[ \widehat{\gamma}^{\calB}_{-\lambda-2\rho} \circ \For \circ \bfD_{\calT}^{\gr}(\fP^{y,\gr}_{\tau_0 w}) \ \cong \ (\widehat{\gamma}^{\calB}_{\lambda}(\fP^y_{\tau_0 w}))^{\vee} \] is in degree $0$. (Here, $\For$ is the functor of \eqref{eq:For-DGCoh-rcap}.) Hence $n=0$, which finishes the proof.\end{proof}

\begin{remark} \begin{enumerate}

\item We will not use the first isomorphism (for simple objects). We only include it for completeness.

\item This proposition, together with Proposition \ref{prop:BMR-duality-Fr}, allows to prove the Frobenius property of regular blocks of $(\Ug)_0$ without refering to the general result in \cite{Be}. Moreover, it describes explicitly the duals of indecomposable projectives. For example, for $\lambda$ in $C_0$, $\iota_{\lambda}$ does not depend on $\lambda$, and is given by $\iota_0(t_{\mu} \cdot v)=t_{-w_0 \mu} \cdot w_0 v w_0$. Hence for $w \in W^0$ we obtain an isomorphism \[ P(w \bullet \lambda)^{\vee} \ \cong \ P(\iota_0(w) \bullet (-w_0 \lambda)) = P(-w_0(w \bullet \lambda)). \] In particular, it follows that \[ \soc \, P(w \bullet \lambda) \cong L( w \bullet \lambda). \] This is of course well-known, see e.g.~\cite[Proposition I.8.13]{Ja}.

\end{enumerate} 

\end{remark}

In the rest of this section we deduce from the geometric Proposition \ref{prop:duality-L-P} algebraic statements about Koszul gradings on regular blocks of $(\Ug)_0$.

\subsection{Grading and the natural anti-isomorphism} \label{ss:grading-anti-isom}

Let again $\lambda \in \bbX$ be regular, and let $y \in W_{\aff}$ be such that $\lambda \in y \bullet C_0$.

Consider the restriction $\wcalD^{\widehat{\lambda}}$ of the sheaf of algebras $\wcalD$ to the formal neighborhood of $\calB^{(1)} \times \{\lambda\}$ in $\wfg^{(1)} \times_{\fh^*{}^{(1)}} \fh^*$, which we identify with the formal neighborhood of $\calB^{(1)}$ in $\wfg^{(1)}$. Consider a decomposition \begin{equation} \label{eq:decomposition-M} \calM^{\lambda} \ = \ \bigoplus_{i \in I} \, \calM^{\lambda}_i \end{equation} of the vector bundle $\calM^{\lambda}$ into indecomposable subbundles. Then it follows from the definitions (see e.g.~\cite{BM}) that the collection $\{ (\calM^{\lambda}_i)^{\star} \otimes_{\bk} \Lambda(\fg^{(1)}), \, i \in I \}$ (where the objects are endowed with a Koszul differential) coincides with the collection of the $\fP_w^y$'s.

As in \S \ref{ss:lkd}, $\wfg^{(1)}$ is endowed with an action of $\Gm$, where $t \in \bk^{\times}$ acts by multiplication by $t^{-2}$ along the fibers of the projection $\wfg^{(1)} \to \calB^{(1)}$. It is explained in \cite[\S 6.3]{Ri} how, starting from a $\Gm$-equivariant structure on $\calM^{\lambda}$ (as a sheaf on $\wfg^{(1)}$), one can produce a grading on the algebra $(\Ug)_0^{\widehat{\lambda}}$, and an equivalence of categories \[ \widetilde{\gamma}^{\calB}_{\lambda} : \ \DGCoh^{\gr}((\wfg \rcap_{\fg^* \times \calB} \calB)^{(1)}) \ \xrightarrow{\sim} \ \calD^b \Mod^{\fig,\gr}_{\lambda}((\Ug)_0) \] which is a ``graded version'' of $\widehat{\gamma}^{\calB}_{\lambda}$ (see \cite[Theorem 6.3.4 and Remark 6.3.5]{Ri}). (Here, $\Mod^{\fig,\gr}_{\lambda}((\Ug)_0)$ is the category of finitely generated graded modules over the graded algebra $(\Ug)_0^{\widehat{\lambda}}$.) To choose a $\Gm$-equivariant structure on $\calM^{\lambda}$, it is enough to choose a $\Gm$-equivariant structure on each of the $\calM^{\lambda}_i$'s. Such a structure is unique up to an internal shift. (Existence is proved in \cite[Lemma 6.3.3]{Ri} or in \cite{BM}, and unicity is obvious.) We choose it in such a way that the collection of $\Gm$-equivariant objects $\{ (\calM^{\lambda}_i)^{\star} \otimes_{\bk} \Lambda(\fg^{(1)}), \, i \in I \}$ coincides with the collection of the $\fP^{y,\gr}_w$'s. We denote by $\calM^{\lambda}_{\gr}$ the resulting $\Gm$-equivariant vector bundle. With this choice, by construction the grading on $(\Ug)_0^{\widehat{\lambda}}$ is the Koszul grading of Theorem \ref{thm:koszul-grading}.

Using this one can prove the regular case of point $(1)$ of the main theorem. Recall the isomorphism $\Phi_{\lambda}$ of \eqref{eq:isom-Phi}.

\begin{prop} \label{prop:antiautomorphism-grading}

Assume $p>h$ is such that Lusztig's conjecture is true.

The isomorphism $\Phi_{\lambda} : (\Ug)_0^{\widehat{\lambda}} \xrightarrow{\sim} ((\Ug)_0^{\widehat{-\lambda-2\rho}})^{\op}$ is an isomorphism of graded algebras, where both algebras are endowed with the Koszul grading given by Theorem {\rm \ref{thm:koszul-grading}}.

\end{prop}

\begin{proof} Consider the automorphism \[ \widetilde{\sigma} : \left\{ \begin{array}{ccc} \wfg^{(1)} \times_{\fh^*{}^{(1)}} \fh^* & \to & \wfg^{(1)} \times_{\fh^*{}^{(1)}} \fh^* \\ (gB,X,\lambda) & \mapsto & (gB,-X,-\lambda-2\rho)
\end{array} \right. . \] As explained in \cite[Lemma 3.0.6(a)]{BMR2}, the natural isomorphism $\calD_{G/U} \cong \calD_{G/U}^{\op}$ (due to triviality of the canonical line bundle on $G/U$) induces an isomorphism of algebras $\wcalD^{\op} \cong \widetilde{\sigma}^* \wcalD$. Moreover, taking global sections this isomorphism gives the natural isomorphism $\Ug \otimes_{\fZ_{\HC}} \rmS(\fh) \cong (\Ug \otimes_{\fZ_{\HC}} \rmS(\fh))^{\op}$. In particular, we obtain an isomorphism of algebras on the formal neighborhood of $\calB^{(1)}$ in $\wfg^{(1)}$: \begin{equation} \label{eq:isom-D-op} \wcalD^{\widehat{\lambda}} \ \cong \ \sigma^* (\wcalD^{\widehat{-\lambda-2\rho}})^{\op}.
\end{equation} Now by \cite[Lemma 3.0.6(b)]{BMR2}, there exists an isomorphism of $\wcalD^{\widehat{\lambda}}$-mo\-dules \begin{equation} \label{eq:isom-M} \calM^{\lambda} \ \cong \ \sigma^* (\calM^{-\lambda-2\rho})^{\star}, \end{equation} where the $\wcalD^{\widehat{\lambda}}$-module structure on the right hand side is induced by isomorphism \eqref{eq:isom-D-op}. It can be easily checked that for any choice of such an isomorphism $\phi$, the induced isomorphism \begin{multline*} \wcalD^{\widehat{\lambda}} \ \cong \ \sheafEnd(\calM^{\lambda}) \ \overset{\phi}{\cong} \ \sigma^* \sheafEnd((\calM^{-\lambda-2\rho})^{\star}) \\ \overset{(-)^{\star}}{\cong} \ \sigma^* \sheafEnd(\calM^{-\lambda-2\rho})^{\op} \ \cong \ \sigma^* (\wcalD^{\widehat{-\lambda-2\rho}})^{\op} \end{multline*} coincides with \eqref{eq:isom-D-op}.

One can choose the decompositions \eqref{eq:decomposition-M} for $\lambda$ and for $-\lambda-2\rho$ to be compatible with isomorphism \eqref{eq:isom-M}. Then by definition and Proposition \ref{prop:duality-L-P}, isomorphism \eqref{eq:isom-M} becomes an isomorphism of $\Gm$-equivariant vector bundles \begin{equation} \label{eq:isom-M-gr} \calM^{\lambda}_{\gr} \ \cong \ \sigma^* (\calM^{-\lambda-2\rho}_{\gr})^{\star} \langle 3N \rangle. \end{equation} The result follows, by construction of the grading on $(\Ug)_0^{\widehat{\lambda}}$ and $(\Ug)_0^{\widehat{-\lambda-2\rho}}$ (see \cite[\S 6.3]{Ri}).\end{proof}

\subsection{Grading and duality} \label{ss:grading-duality}

Using Proposition \ref{prop:antiautomorphism-grading}, one can define a duality functor \[ (-)^{\vee} : \Mod^{\fig,\gr}_{\lambda}((\Ug)_0) \xrightarrow{\sim} \Mod^{\fig,\gr}_{-\lambda-2\rho}((\Ug)_0)^{\op}. \] Let us give a geometric description of this functor. (The proof is an adaptation of that of \cite[Proposition 3.0.9]{BMR2}.)

\begin{prop} \label{prop:graded-duality}

Assume $p>h$ is such that Lusztig's conjecture is true.

Let $\lambda \in \bbX$ be regular. Then the following diagram commutes: \[ \xymatrix@C=60pt{ \DGCoh^{\gr}((\wfg
\, \rcap_{\fg^* \times \calB} \, \calB)^{(1)}) \ar[d]_-{\widetilde{\gamma}_{\lambda}^{\calB}}^-{\wr} \ar[r]^-{\bfD^{\gr}_{\calT}} & \DGCoh^{\gr}((\wfg
\, \rcap_{\fg^* \times \calB} \, \calB)^{(1)}) \ar[d]^-{\widetilde{\gamma}_{-\lambda-2\rho}^{\calB}}_-{\wr} \\ \calD^b \Mod^{\fig,\gr}_{\lambda}((\calU \fg)_0) \ar[r]^-{(-)^{\vee} \langle 2 N \rangle} & \calD^b \Mod^{\fig,\gr}_{-\lambda-2\rho}((\calU \fg)_0). } \]

\end{prop}

\begin{proof} First we begin with an easy lemma, whose proof is similar to that of \cite[Lemma 3.0.1]{BMR2}.

\begin{lem} \label{lem:duality-g^*}

Let $M$ be a finite complex of finite dimensional $\rmS(\fg)$-mo\-dules (or equivalently quasi-coherent $\calO_{\fg^*}$-modules). Then there exists a functorial isomorphism, in the derived category of $\rmS(\fg)$-modules \[ \bbD_{\fg^*}(M) \ \cong \ M^* [-d] \langle -2d \rangle. \]

\end{lem}

Now, let $\pi : \wfg^{(1)} \to \fg^*{}^{(1)}$ be the natural morphism. Then, as Grothen\-dieck-Serre duality commutes with proper direct images, we have an isomorphism \begin{equation} \label{eq:duality-pi_*} R\pi_* \circ \bbD_{\wfg} \ \cong \ \bbD_{\fg^*} \circ R\pi_* \langle 2N \rangle.
\end{equation}

Let $\calF \in \DGCoh^{\gr}((\wfg \, \rcap_{\fg^* \times \calB} \, \calB)^{(1)})$. By definition and Lemma \ref{lem:duality-g^*} we have an isomorphism \begin{align*} \widetilde{\gamma}^{\calB}_{\lambda}(\calF)^{\vee} \ & = \ R\Gamma(\calM^{\lambda}_{\gr} \otimes_{\calO_{\wfg^{(1)}}} Rp_* \calF)^* \\ & \cong \ \bbD_{\fg^*} \circ R\pi_*(\calM^{\lambda}_{\gr} \otimes_{\calO_{\wfg^{(1)}}} Rp_* \calF) [d] \langle 2d \rangle. \end{align*} Using \eqref{eq:duality-pi_*} we deduce \[ \widetilde{\gamma}^{\calB}_{\lambda}(\calF)^{\vee} \ \cong \ R\pi_* \circ \bbD_{\wfg}(\calM^{\lambda}_{\gr} \otimes_{\calO_{\wfg^{(1)}}} Rp_* \calF) [d] \langle 2d-2N \rangle. \] Now there is a natural isomorphism \[ \bbD_{\wfg}(\calM^{\lambda}_{\gr} \otimes_{\calO_{\wfg^{(1)}}} \calF) \ \cong \ (\calM^{\lambda}_{\gr})^{\star} \otimes_{\calO_{\wfg^{(1)}}} \bbD_{\wfg}(Rp_* \calF). \] Hence, using isomorphism \eqref{eq:isom-M-gr} we obtain \[ \widetilde{\gamma}^{\calB}_{\lambda}(\calF)^{\vee} \ \cong \ R\Gamma \bigl( \calM_{\gr}^{-\lambda-2\rho} \otimes_{\calO_{\wfg^{(1)}}} \sigma^* \bbD_{\wfg}(Rp_* \calF) \bigr) [d] \langle 2d - 5N \rangle. \] Using finally the isomorphism $\bfD_{\wfg}^{\gr} \circ Rp_* \cong Rp_* \circ \bfD_{\calT}^{\gr}$, we obtain \begin{multline*} \widetilde{\gamma}^{\calB}_{\lambda}(\calF)^{\vee} \ \cong \ R\Gamma \bigl( \calM_{\gr}^{-\lambda-2\rho} \otimes_{\calO_{\wfg^{(1)}}} Rp_* (\bfD_{\calT}^{\gr} \calF) \bigr) \langle - 2 N \rangle \\ \cong \ \widetilde{\gamma}^{\calB}_{-\lambda-2\rho} \circ \bfD_{\calT}^{\gr}(\calF) \langle -2N \rangle . \end{multline*} One can check that all these isomorphisms are isomorphisms in the derived category of $(\Ug)_0$-modules, which concludes the proof. \end{proof}

\subsection{Duality and the regular representation}

Now we can prove point $(2)$ of the regular case of the main Theorem.

\begin{prop} \label{prop:duality-regular-rep}

Assume $p>h$ is such that Lusztig's conjecture is true.

Let $\lambda \in \bbX$ be regular. There exists an isomorphism of graded $(\Ug)_0^{\widehat{\lambda}}$-modules \[ (\Ug)_0^{\widehat{\lambda}} \ \cong \ ((\Ug)_0^{\widehat{-\lambda-2\rho}})^{\vee} \langle 2 N \rangle. \]

\end{prop}

\begin{proof} This follows from Proposition \ref{prop:graded-duality}, using the fact that \[ (\Ug)_0^{\widehat{\lambda}} \ \cong \ \widetilde{\gamma}^{\calB}_{\lambda} \bigl( (\calM^{\lambda}_{\gr})^{\star} \otimes_{\bk} \Lambda(\fg^{(1)}) \bigr), \] and the natural isomorphism \[ (\calM^{\lambda}_{\gr})^{\star} \otimes_{\bk} \Lambda(\fg^{(1)}) \ \cong \ \bfD_{\calT}^{\gr}((\calM^{-\lambda-2\rho}_{\gr})^{\star} \otimes_{\bk} \Lambda(\fg^{(1)})) \] which follows from Proposition \ref{prop:duality-L-P} (see the proof of Proposition \ref{prop:antiautomorphism-grading}). \end{proof}

\begin{remark}

\begin{enumerate}

\item It follows in particular from this proposition that the maximal degree in the Koszul grading of $(\Ug)_0^{\widehat{\lambda}}$ is $2N$. It follows also that the Poincaré polynomial $P_{\lambda}$ of this grading satisfies $P_{\lambda}(t^{-1})=t^{-2N} P_{\lambda}(t)$. (The former fact can also be derived directly from \cite{Ri}, using the property that the homological dimension of $(\Ug)^{\lambda}$ is $2N$.)

\item In fact the algebra $(\Ug)_0$ is a \emph{symmetric} algebra, see \cite{Sc, FP}, hence the same is true for $(\Ug)_0^{\widehat{\lambda}}$. In other words, there exists an isomorphism of $(\Ug)_0^{\widehat{\lambda}}$-bimodules $(\Ug)_0^{\widehat{\lambda}} \ \cong \ ((\Ug)_0^{\widehat{-\lambda-2\rho}})^{\vee}$. One can easily check that the isomorphism of Proposition \ref{prop:duality-regular-rep} is an isomorphism of graded $(\Ug)_0^{\widehat{\lambda}}$-bimodules.

\end{enumerate}

\end{remark}

\section{Koszul duality and ordinary duality: singular case} \label{sec:koszul-duality-singular}

In this section we prove analogues of the results of Section \ref{sec:koszul-duality-regular} for singular blocks. Most of the proofs are similar to those for the regular case, hence we omit them.

\subsection{Geometry}

Recall the constructions of \S \ref{ss:duality}. With the data $X=\calP^{(1)}$, $E=(\fg^* \times \calP)^{(1)}$, $F=\wcalN_{\calP}^{(1)}$, we obtain a duality functor \[ \bbD_{\calT,\calP} : \DGCoh^{\gr}((\wfg_{\calP} \, \rcap_{\fg^* \times \calP} \, \calP)^{(1)}) \xrightarrow{\sim} \DGCoh^{\gr}((\wfg_{\calP} \, \rcap_{\fg^* \times \calP} \, \calP)^{(1)}), \] which is a ``graded version'' of the functor $\bbD_{\calT,\calP}^0$ of \S \ref{ss:geometric-duality}. We define \[ \bfD_{\calT,\calP}^{\gr} \ := \ \sigma^* \bbD_{\calT,\calP} \langle -3 N \rangle. \] Then, one easily checks that we have an isomorphism \begin{equation} \label{eq:D-pi_calP} \bfD_{\calT}^{\gr} \circ L(\widehat{\pi}_{\calP,\Gm})^* \langle 2N -2N_{\calP} \rangle \ \cong \ L(\widehat{\pi}_{\calP,\Gm})^* \circ \bfD_{\calT,\calP}^{\gr}. \end{equation}

Let us fix a weight $\mu \in \bbX$ and a standard parabolic $P$ as in Theorem \ref{thm:thmBMR}$\rii$. Then $\mu-2\rho$ satisfies the same assumption, for the same parabolic subgroup (see Remark \ref{rk:singular-weights}). There exist a unique $y \in W_{\aff}$ such that the alcove $y \bullet C_0$ contains $\mu$ in its closure, and contains also a weight $\lambda$ orthogonal to all the roots of the Levi of $P$. Let also $z \in W_{\aff}$ be defined similarly, for $-\mu-2\rho$ instead of $\mu$. Then we have the following.

\begin{prop} \label{prop:duality-P-singular}

For $w \in W^0_{\mu}$, there is an isomorphism \[ \bfD_{\calT,\calP}^{\gr}(\fP^{y,\gr}_{\calP,w}) \ \cong \ \fP^{z,\gr}_{\calP,\tau_0 \iota_{\lambda}(\tau_0 w)}. \]

\end{prop}

\begin{proof} By Proposition \ref{prop:duality-L-P}, there is an isomorphism \[ \bfD_{\calT}^{\gr}(\fP^{y,\gr}_{w}) \ \cong \ \fP^{z,\gr}_{\tau_0 \iota_{\lambda}(\tau_0 w)}. \] Using equation \eqref{eq:fP_calP}, we deduce  \[ \bfD_{\calT}^{\gr} \circ L(\widehat{\pi}_{\calP,\Gm})^* \, \fP^{y,\gr}_{\calP,w} \ \cong \ L(\widehat{\pi}_{\calP,\Gm})^* \, \fP^{z,\gr}_{\calP,\tau_0 \iota_{\lambda}(\tau_0 w)} \langle N_{\calP} - N \rangle. \] Then, using \eqref{eq:D-pi_calP}, we obtain \[ L(\widehat{\pi}_{\calP,\Gm})^* \circ \bfD_{\calT,\calP}^{\gr} (\fP^{y,\gr}_{\calP,w}) \ \cong \ \fP^{z,\gr}_{\calP,\tau_0 \iota_{\lambda}(\tau_0 w)} \langle N - N_{\calP} \rangle. \] The result follows, by Lemma \ref{lem:characterization-fP_calP}.\end{proof}

\begin{remark}

\begin{enumerate}
\item One can prove a similar statement for the objects ``$\calL^{y,\gr}_{\calP,w}$'' of \cite[Theorem 10.2.4]{Ri}. As we do not need this statement (not even for the proof of Proposition \ref{prop:duality-P-singular}), we omit it.
\item Again, this proposition allows to prove the Frobenius property of singular blocks of $(\Ug)_0$ without refering to the general result of \cite{Be}. It also gives a description of duals of indecomposable projectives.
\end{enumerate}

\end{remark}

\subsection{Grading and the natural anti-isomorphism} \label{ss:grading-anti-isom-singular}

As in \S \ref{ss:grading-anti-isom} we choose a $\Gm$-equivariant structure on the vector bundle $\calM^{\mu}_{\calP}$ (as a sheaf on the formal neighborhood of $\calP^{(1)}$ in $\wfg_{\calP}^{(1)}$) which is compatible with the $\Gm$-equivariant structures on the $\fP^{y,\gr}_{\calP,w}$'s. Then we have an equivalence of categories \[ \widetilde{\gamma}^{\calP}_{\mu} : \ \DGCoh^{\gr}((\wfg \rcap_{\fg^* \times \calB} \calB)^{(1)}) \ \xrightarrow{\sim} \ \calD^b \Mod^{\fig,\gr}_{\mu}((\Ug)_0), \] where the grading on the algebra $(\Ug)_0^{\widehat{\mu}}$ is the Koszul grading provided by Theorem \ref{thm:koszul-grading-singular} (see \cite[\S 10.2]{Ri}).

One can make the same constructions for the weight $-\mu-2\rho$. Then, one can prove the following singular analogue of Proposition \ref{prop:antiautomorphism-grading}.

\begin{prop} \label{prop:antiautomorphism-grading-singular}

Assume $p>h$ is such that Lusztig's conjecture is true, and conditions \eqref{eq:highervanishing} and \eqref{eq:phisurjective} are satisfied.

The isomorphism $\Phi_{\mu} : (\Ug)_0^{\widehat{\mu}} \to ((\Ug)_0^{\widehat{-\mu-2\rho}})^{\op}$ is an isomorphism of graded algebras, where both algebras are endowed with the Koszul grading given by Theorem {\rm \ref{thm:koszul-grading}}.

\end{prop}

\subsection{Grading and duality} \label{ss:grading-duality-singular}

Using Proposition \ref{prop:antiautomorphism-grading-singular}, one can define a duality functor \[ (-)^{\vee} : \Mod^{\fig,\gr}_{\mu}((\Ug)_0) \xrightarrow{\sim} \Mod^{\fig,\gr}_{-\mu-2\rho}((\Ug)_0)^{\op}. \] As in Proposition \ref{prop:graded-duality}, one can give a geometric description of this functor.

\begin{prop} \label{prop:graded-duality-singular}

Assume $p>h$ is such that Lusztig's conjecture is true, and conditions \eqref{eq:highervanishing} and \eqref{eq:phisurjective} are satisfied.

The following diagram commutes: \[ \xymatrix@C=60pt{ \DGCoh^{\gr}((\wfg_{\calP}
\, \rcap_{\fg^* \times \calP} \, \calP)^{(1)}) \ar[d]_-{\widetilde{\gamma}_{\mu}^{\calP}}^-{\wr} \ar[r]^-{\bfD^{\gr}_{\calT,\calP}} & \DGCoh^{\gr}((\wfg_{\calP}
\, \rcap_{\fg^* \times \calP} \, \calP)^{(1)}) \ar[d]^-{\widetilde{\gamma}_{-\mu-2\rho}^{\calP}}_-{\wr} \\ \calD^b \Mod^{\fig,\gr}_{\mu}((\calU \fg)_0) \ar[r]^-{(-)^{\vee} \langle 2 N_{\calP} \rangle} & \calD^b \Mod^{\fig,\gr}_{-\mu-2\rho}((\calU \fg)_0). } \]

\end{prop}

\subsection{Duality and the regular representation} 

Finally, one can deduce the following.

\begin{prop} \label{prop:duality-regular-rep-singular}

Assume $p>h$ is such that Lusztig's conjecture is true, and conditions \eqref{eq:highervanishing} and \eqref{eq:phisurjective} are satisfied.

There exists an isomorphism of graded $(\Ug)_0^{\widehat{\mu}}$-modules \[ (\Ug)_0^{\widehat{\mu}} \ \cong \ ((\Ug)_0^{\widehat{-\mu-2\rho}})^{\vee} \langle 2 N_{\calP} \rangle. \]

\end{prop}

\begin{remark}

\begin{enumerate}

\item It follows in particular from this proposition that the maximal degree in the Koszul grading of $(\Ug)_0^{\widehat{\mu}}$ is $2N_{\calP}$. It follows also that the Poincaré polynomial $P_{\mu}$ of this grading satisfies $P_{\mu}(t^{-1})=t^{-2N_{\calP}} P_{\mu}(t)$.

\item One can easily check that the isomorphism of Proposition \ref{prop:duality-regular-rep-singular} is an isomorphism of graded $(\Ug)_0^{\widehat{\mu}}$-bimodules.

\end{enumerate}

\end{remark}

\section{Example: $\SL(2)$} \label{sec:example}

In this section we set $G=\SL(2)$, and we assume $p>2$. There is a natural isomorphism $\bbX \cong \bbZ$, such that $\rho=1$. Also, we have $\calB \cong \bbP^1$. For simplicity, we omit Frobenius twists in this section. Such a twist should appear on every variety we consider.

\subsection{Regular blocks}

In this case the regular blocks are those of the weights $0, \cdots, \frac{p-3}{2}$. Let us fix $\lambda \in \{ 0, \cdots, \frac{p-3}{2} \}$. The simple objects in $\Mod^{\fig}_{\lambda}((\calU \fg)_0)$ are $L(\lambda)$ (of dimension $\lambda+1$) and $L(p-2-\lambda)$ (of dimension $p-1-\lambda$). It is easy to check (see \cite[Corollary 7.2.6]{Ri}) that \[ \fL_1=\calO_{\bbP^1}(-1), \qquad \fL_{\tau_0} = \calO_{\bbP^1}(-2)[1] \] (considered as sheaves on the zero section of $\wcalN$). Hence the Koszul grading on $(\Ug)_0^{\widehat{\lambda}}$ is given by the isomorphism \[ (\Ug)_0^{\widehat{\lambda}} \ \cong \ \bigoplus_{n \in \bbZ} \, \Ext^n_{\wcalN}(\fL_1^{\oplus \lambda+1} \oplus \fL_{\tau_0}^{\oplus p-1-\lambda}, \, \fL_1^{\oplus \lambda+1} \oplus \fL_{\tau_0}^{\oplus p-1-\lambda}). \]

One can easily check, using the Koszul resolution \[ \calO_{\wcalN}(2) \hookrightarrow \calO_{\wcalN} \twoheadrightarrow \calO_{\bbP^1}, \] that we have \begin{align*} \Ext^*_{\wcalN}(\calO_{\bbP^1},\, \calO_{\bbP^1}) \ & \cong \ \bk \oplus \bk[-2], \\ \Ext^*_{\wcalN}(\calO_{\bbP^1}(-1),\, \calO_{\bbP^1}) \ & \cong \ V^*, \\ \Ext^*_{\wcalN}(\calO_{\bbP^1},\, \calO_{\bbP^1}(-1)) \ & \cong \ V[-2], \end{align*} where $V$ is a $2$-dimensional $\bk$-vector space. Hence we obtain an isomorphism \begin{multline*} (\Ug)_0^{\widehat{\lambda}} \ \cong \ \bigl( \mathrm{Mat}_{\lambda+1}(\bk) \otimes_{\bk} (\bk \oplus \bk[-2]) \bigr) \oplus \bigl( \mathrm{Mat}_{p-1-\lambda}(\bk) \otimes_{\bk} (\bk \oplus \bk[-2]) \bigr) \\ \oplus \bigl( V \otimes_{\bk} \mathrm{Mat}_{\lambda+1,p-1-\lambda}(\bk) [-1] \bigr) \oplus \bigl( V^* \otimes_{\bk} \mathrm{Mat}_{p-1-\lambda,\lambda+1}(\bk) [-1] \bigr). \end{multline*} The product is given by the matrix multiplication, together with the natural paring $V[-1] \times V^*[-1] \to \bk[-2]$.

In this case $w_0=-1$, hence all the blocks are self-dual. The identification of $(\Ug)_0^{\widehat{\lambda}}$ with (the shift of) its dual is given by the natural identification of the dual of $\mathrm{Mat}_{n,m}(\bk)$ with $\mathrm{Mat}_{m,n}(\bk)$, via the trace.

In particular, all regular blocks are of dimension $2 p^2$, and they are Morita equivalent to the path algebra of the quiver 
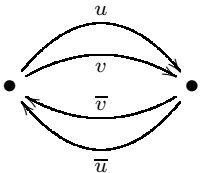
\[ \xymatrix@C=2cm{ \bullet \ar@/^2pc/[r]^-{u} \ar@/^1pc/[r]_-{v} & \bullet \ar@/^2pc/[l]^-{\overline{u}} \ar@/^1pc/[l]_-{\overline{v}} } \] with relations \begin{align*} \overline{u} v = \overline{v} u = 0, \quad & \overline{u} u = \overline{v} v, \\ u \overline{v} = v \overline{u} = 0, \quad & u \overline{u} = v \overline{v}. \end{align*} The grading is obtained by assigning each edge the degree $1$.

\subsection{Singular blocks}

The only singular block is that of $-1$. And the only simple object in the category $\Mod^{\fig}_{-1}((\calU \fg)_0)$ is $L(p-1)$, of dimension $p$. Moreover, this category is semisimple. Hence we have an isomorphism of graded rings \[ (\Ug)_0^{\widehat{-1}} \ \cong \ \mathrm{Mat}_p(\bk). \] Again, the identification of $(\Ug)_0^{\widehat{-1}}$ with its dual is given by the natural identification of $\mathrm{Mat}_p(\bk)$ with its dual.

\end{document}